\documentclass{amsart}

\usepackage{etex}
\usepackage{amsmath, amssymb}
\usepackage{array}
\usepackage[frame,cmtip,arrow,matrix,line,graph,curve]{xy}
\usepackage{graphpap, color, paralist, pstricks}
\usepackage[mathscr]{eucal}
\usepackage[pdftex]{graphicx}
\usepackage[pdftex,colorlinks,backref=page,citecolor=blue]{hyperref}
\usepackage{pifont}
\usepackage{cancel}
\usepackage{mathtools}
\usepackage{tikz}

\newcommand{\mA}[0]{\mathfrak{A}}
\newcommand{\mB}[0]{\mathfrak{B}}
\newcommand{\mC}[0]{\mathfrak{C}}
\newcommand{\mD}[0]{\mathfrak{D}}
\newcommand{\mE}[0]{\mathfrak{E}}
\newcommand{\mF}[0]{\mathfrak{F}}
\newcommand{\mG}[0]{\mathfrak{G}}
\newcommand{\mH}[0]{\mathfrak{H}}
\newcommand{\mI}[0]{\mathfrak{I}}
\newcommand{\mJ}[0]{\mathfrak{J}}

\newcommand{\mL}[0]{\mathfrak{L}}

\newcommand{\mN}[0]{\mathfrak{N}}

\newcommand{\mQ}[0]{\mathfrak{Q}}
\newcommand{\mR}[0]{\mathfrak{R}}
\newcommand{\mS}[0]{\mathfrak{S}}

\newcommand{\mV}[0]{\mathfrak{V}}
\newcommand{\mX}[0]{\mathfrak{X}}
\newcommand{\mZ}[0]{\mathfrak{Z}}

\newcommand{\pr}{\mathbb{P}}

\newcommand{\E}[0]{\mathbb{E}}

\newcommand{\beq}[1]{\begin{equation}\label{#1}}
\newcommand{\enq}[0]{\end{equation}}

\newcommand{\bn}[0]{\bigskip\noindent}
\newcommand{\mn}[0]{\medskip\noindent}
\newcommand{\nin}[0]{\noindent}

\newcommand{\sub}[0]{\subseteq}
\newcommand{\sm}[0]{\setminus}
\renewcommand{\dots}[0]{,\ldots,}

\newcommand{\ov}[0]{\overline}

\newcommand{\less}[0]{~~\mbox{\raisebox{-.6ex}{$\stackrel{\textstyle{<}}{\sim}$}}~~}
\newcommand{\more}[0]{~~\mbox{\raisebox{-.9ex}{$\stackrel{\textstyle{>}}{\sim}$}}~~}

\newcommand{\A}[0]{{\mathcal A}}
\newcommand{\B}[0]{{\mathcal B}}

\newcommand{\D}[0]{{\mathcal D}}
\newcommand{\eee}[0]{{\mathcal E}}
\newcommand{\f}[0]{{\mathcal F}}

\newcommand{\g}[0]{{\mathcal G}}
\newcommand{\h}[0]{{\mathcal H}}
\newcommand{\I}[0]{{\mathcal I}}
\newcommand{\J}[0]{{\mathcal J}}
\newcommand{\K}[0]{{\mathcal K}}
\newcommand{\Ll}[0]{{\mathcal L}}

\newcommand{\N}[0]{{\mathcal N}}

\newcommand{\R}[0]{{\mathcal R}}
\newcommand{\sss}[0]{{\mathcal S}}
\newcommand{\T}[0]{{\mathcal T}}
\newcommand{\U}[0]{{\mathcal U}}

\newcommand{\W}[0]{{\mathcal W}}
\newcommand{\X}[0]{{\mathcal X}}

\newcommand{\ZZZ}[0]{{\mathcal Z}}
\newcommand{\Z}[0]{{\mathcal Z}}
\newcommand{\YYY}[0]{{\mathcal Y}}

\newcommand{\bgs}[0]{{\boldsymbol \gs}}

\newcommand{\bF}[0]{\boldsymbol\f}

\newcommand{\bG}[0]{\boldsymbol\g}
\newcommand{\bH}[0]{\boldsymbol\h}

\newcommand{\bJ}[0]{\boldsymbol\J}
\newcommand{\bN}[0]{\boldsymbol\N}
\newcommand{\bT}[0]{\boldsymbol\T}
\newcommand{\bU}[0]{\boldsymbol\U}

\newcommand{\bX}[0]{\boldsymbol\X}

\newcommand{\bZ}[0]{\boldsymbol \Z}

\newcommand{\ra}[0]{\rightarrow}
\newcommand{\Ra}[0]{\Rightarrow}

\newcommand{\FF}[0]{{\bf F}}

\newcommand{\uu}[0]{U}

\newcommand{\ww}{\mbox{{\sf w}}}

\newcommand{\0}[0]{\emptyset}

\newcommand{\C}[2]{{{#1}\choose{{#2}}}}
\newcommand{\Cc}[0]{\tbinom}
\newcommand{\ga}[0]{\alpha }
\newcommand{\gb}[0]{\beta }
\newcommand{\gc}[0]{\gamma }
\newcommand{\gd}[0]{\delta }
\newcommand{\gD}[0]{\Delta }
\newcommand{\gG}[0]{\Gamma }

\newcommand{\gl}[0]{\lambda }
\newcommand{\gL}[0]{\Lambda}
\newcommand{\go}[0]{\omega}
\newcommand{\gO}[0]{\Omega}

\newcommand{\gs}[0]{\sigma}
\newcommand{\gS}[0]{\Sigma}
\newcommand{\gz}[0]{\zeta}
\newcommand{\eps}[0]{\varepsilon }
\newcommand{\vt}[0]{\vartheta}
\newcommand{\vs}[0]{\varsigma}
\newcommand{\vr}[0]{\varrho}
\newcommand{\vp}[0]{\varphi}

\def\maxr{{\rm maxr \,\,}}

\newcommand{\Rastar}[0]{~~\mbox{\raisebox{-.05ex}{$\stackrel{*}{\Longrightarrow}$}}~~}

\newcommand{\barvp}[0]{\ov{\vp}}
\newcommand{\mmodels}[0]{\models}
\newcommand{\epp}[0]{\sf{c}}
\newcommand{\JJJ}[0]{\J}
\newcommand{\vrr}[0]{\vr}

\newcommand{\tdI}[0]{\tilde{\mI}}
\newcommand{\ud}[0]{\underline{d}}
\newcommand{\ph}[0]{\pr_h}
\newcommand{\pu}[0]{\pr_u}
\newcommand{\prh}[1][]{\pr_h}

\newtheorem{thm}{Theorem}[section]
\newtheorem{prop}[thm]{Proposition}

\newtheorem{lemma}[thm]{Lemma}
\newtheorem{obs}[thm]{Observation}
\newtheorem{cor}[thm]{Corollary}

\theoremstyle{definition}

\begin{document}

\title{Hitting times for Shamir's Problem}

\author{Jeff Kahn}
\thanks{Supported by NSF Grants DMS1501962 and DMS1954035, BSF Grant 2014290,
and a Simons Fellowship.}
\email{jkahn@math.rutgers.edu}
\address{Department of Mathematics, Rutgers University \\
Hill Center for the Mathematical Sciences \\
110 Frelinghuysen Rd.\\
Piscataway, NJ 08854-8019, USA}

\begin{abstract}

For fixed $r\geq 3$ and $n$ divisible by $r$, let 
$\bH=\bH^r_{n,M}$ be the random $M$-edge $r$-graph on
$V=\{1\dots n\}$; that is, $\bH$ is chosen uniformly from the $M$-subsets of $\K:=\C{V}{r}$ ($:= \{\mbox{$r$-subsets of $V$}\}$).
\emph{Shamir's Problem} (circa 1980) asks, roughly,
\begin{center}
\emph{for what $M=M(n)$ is $\bH$ likely to contain a perfect matching}
\end{center}
(that is, $n/r$ disjoint $r$-sets)?

In 2008 Johansson, Vu and the author showed that this is true for $M>C_rn\log n$.  
More recently the author proved the
asymptotically correct version of that result:  for fixed $C> 1/r$ and $M> Cn\log n$,
\[
\pr(\bH ~\mbox{\emph{contains a perfect matching}})\ra 1 \,\,\,
\mbox{\emph{as} $n\ra\infty$}.
\]

The present work completes a proof, begun in that recent paper,
of the definitive ``hitting time" statement:

\mn
\textbf{Theorem.}
\emph{If $A_1, \ldots ~$ is a uniform permutation of $\K$,
$\bH_t=\{A_1\dots A_t\}$, and
\[
T=\min\{t:A_1\cup \cdots\cup A_t=V\},
\]
then
$\pr(\bH_T ~\mbox{contains a perfect matching})\ra 1 \,\,\,
\mbox{\emph{as} $n\ra\infty$}$.}

\end{abstract}

\maketitle

\section{Introduction}\label{Intro}

A (simple) \emph{r-graph} (or \emph{r-uniform hypergraph}) is
a set $\h$ of $r$-subsets (\emph{edges}) of a
\emph{vertex} set $V=V(\h)$;
a \emph{matching} of such an $\h$ is a set of disjoint edges;
and a \emph{perfect matching} (p.m.)
is a matching of size $|V|/r$.
Write $\bH^r_{n,M}$ for the random $M$-edge $r$-graph on
$[n]:=\{1\dots n\}$; that is, $\bH^r_{n,M}$ is chosen uniformly from the $M$-subsets of $\K:=\C{[n]}{r}$.
(Usage notes are collected at the end of this section.)

This paper completes a proof, begun in \cite{AsSh}, of the definitive answer to 
\emph{Shamir's Problem,} which asks, roughly:
for fixed $r$ and $n$ ranging over (large) multiples $r$,
\begin{center}
\emph{for what $M$ is $\bH^r_{n,M}$ likely to contain a perfect matching?}
\end{center}

\nin
In what follows we work with a fixed $r$ and omit it
from our notation---so $\bH^r_{n,M}$ becomes $\bH_{n,M}$---and restrict to
$n$ divisible by $r$.

The story of Shamir's Problem has been told at some length in \cite{AsSh} 
and we will be briefer here.
The problem first appeared in print in
\cite{Erdos}, where Erd\H{o}s says he
heard it from Eli Shamir in 1979,
and,
following initial results of Schmidt and Shamir
in \cite{SS}, became one of the most intensively studied
questions in probabilistic combinatorics; for example, \cite[Section 4.3]{JLR}
calls Shamir's Problem and its graph factor
analogue (see below)
``two of the most challenging, unsolved problems
in the theory of random structures."

For precise statements define \emph{the}
threshold for containing a perfect matching, denoted $M_c=M_c(n)$, to be the least
the least $M$ for which
\[
\pr(\mbox{$\bH_{n,M}$
contains a perfect matching})\geq 1/2.
\]
(This is also \emph{a} threshold in the original sense of Erd\H{o}s and R\'enyi~\cite{ER}; 
see \cite{BTth} or \cite[Theorem 1.24]{JLR}.)

A natural guess---though not recognized as such in \cite{Erdos,SS}---is that
\emph{in the random setting} 
the main obstacle to existence of a perfect matching is isolated vertices
(vertices not in any edges), which typically disappear when $M\approx (n/r)\log n$.
There are three progressively stronger versions of this intuition that one
might hope to establish.  The first, called 
\emph{Erd\H{o}s-R\'enyi threshold}, involves the order of magnitude of $M_c$:
\begin{thm}\label{jkv}
For each $r$ there is $C_r$ such that
if $M>C_rn\log n$ then $\bH_{n,M}$ contains a perfect matching
w.h.p.\footnote{``with high probability," meaning with probability tending to 1 as $n\ra \infty$}
\end{thm}
\nin
(Equivalently,
$M_c = \Theta(n\log n)$, where the implied constant depends on $r$.)
This was shown in \cite{JKV},
with best earlier progress in
\cite{FJ} and \cite{Kim}.
(See also \cite[Sec.\ 13.2]{Frieze-Karonski} for an exposition.)

The more precise second and third versions---\emph{asymptotics of the threshold} 
and \emph{hitting time}---are:

\begin{thm}\label{ThmX}
For fixed $C> 1/r$ and $M> Cn\log n$, $\bH_{n,M}$
contains a perfect matching w.h.p.
\end{thm}
\nin
(Equivalently, $M_c(n) \sim (n/r)\log n$.)
\begin{thm}\label{ThmY}
If $A_1, \ldots ~$ is a uniform permutation of $\K$,
$\bH_t=\{A_1\dots A_t\}$, and
\[   
T=\min\{t:A_1\cup \cdots\cup A_t=V\},
\]  
then
$\bH_T$ contains a perfect matching w.h.p.
\end{thm}
\nin
(Here $T$
is the aforementioned \emph{hitting time}.
It is easy to see that Theorem~\ref{ThmY} implies Theorem~\ref{ThmX}.)
For $r=2$, 
Theorems~\ref{ThmX} and \ref{ThmY}
were shown by
Erd\H{o}s and R\'enyi \cite{ERPM} and
Bollob\'as and Thomason \cite{BT} respectively.

Theorem~\ref{ThmX} was first formally conjectured (in a stronger form
corresponding to what's shown in \cite{ERPM}) in 
\cite{CFMR} and 
Theorem~\ref{ThmY} was proposed in \cite{JKV},
though each was probably considered plausible by the time it
was recorded.  
(That Theorem~\ref{jkv} was apparently not on the radar in 
\cite{Erdos,SS}---Erd\H{o}s specifically says he has no idea what to expect for Shamir's
Problem---seems odd in view of \cite{ERPM},
but perhaps suggests that the above results were initially thought too much to expect.)

The predecessor, \cite{AsSh}, of the present work proved Theorem~\ref{ThmX} and
began a proof of Theorem~\ref{ThmY} whose completion is our main objective here.
The proof proceeds by way of 
a reduction---given in \cite{AsSh}---to a conditional 
version of Theorem~\ref{ThmX} (Theorem~\ref{ThmZ} below).
The proof of the conditional statement is similar to the proof of Theorem~\ref{ThmX};
but the conditioning---on a low probability event---makes even
formerly routine points tricky to deal with, and
the point of the separate proof of Theorem~\ref{ThmX} was to show the 
structure of the argument unencumbered by these extra difficulties.

\mn
\emph{Graph factors}
(briefly; see \cite{JKV,AsSh} for a little more).
Recall that, for graphs $H$ and $G$, an \emph{$H$-factor} of $G$ is a
collection of copies of $H$ in $G$ 
whose vertex sets partition $V(G)$.
The graph factor counterpart of Shamir's Problem asks (roughly): for a fixed $H$,
\emph{when is the random graph $G_{n,M}$ likely to contain an $H$-factor?}
This was first suggested (for $H=K_3$) by
Ruci\'nski \cite{Rucinski2}.

The factor analogue of Theorem~\ref{jkv} was shown in \cite{JKV} for
\emph{strictly balanced} $H$ (more or less those $H$'s
for which one expects it to hold; see \cite[Conjecture~1.1]{JKV} for what should
be true in general).
For certain nice $H$'s---e.g.\ cliques---beautiful coupling arguments of Riordan 
and Heckel \cite{Riordan,Heckel} derive the factor versions of
Theorems~\ref{jkv} and \ref{ThmX} \emph{from} their Shamir versions,
a connection that seems unlikely to extend to Theorem~\ref{ThmY}.
As suggested in \cite{AsSh}, I expect 
that the work there and here extends to factors, though, at least for what
we do here,
this looks fairly excruciating absent some simplification of the material below.

\mn

In recent work, Frankston, Narayanan, Park and the author \cite{FKNP}
used a new approach inspired by \cite{ALWZ} to prove a general threshold result
(a relaxation, suggested by Talagrand \cite{Talagrand}, of a still open conjecture of Kalai and the  
author \cite{KK}) that easily implies Theorem~\ref{jkv} and much more.
It is, however, hard to imagine Theorem~\ref{ThmX} (\emph{a fortiori} 
Theorem~\ref{ThmY}) being proved along similar lines.
It would be very interesting to decide whether the approach of \cite{FKNP}
can recover the factor version of Theorem~\ref{jkv} proved in \cite{JKV}.

As in \cite{JKV,AsSh}, the proofs of Theorems~\ref{ThmX} and \ref{ThmY}
depend crucially on establishing stronger counting versions;
thus for Theorem~\ref{ThmY}, 
with $\Phi(\h)$ denoting the number of perfect matchings of $\h$,
we show:
\begin{thm}\label{ThmY'}
For $\bH_t$ and $T$ as in Theorem~\ref{ThmY}, w.h.p.
\beq{TY'bd}
\Phi(\bH_T) > \left[e^{-(r-1)}\log n\right]^{n/r}e^{-o(n)} .
\enq
\end{thm}
\nin
(Up to a subexponential factor, the right-hand side of \eqref{TY'bd}
is the expected value of its left-hand side.)

Our assignment here is to prove the following conditional statement,
which, as shown in \cite[Section 10]{AsSh}, implies Theorem~\ref{ThmY'}.
(The same reduction gets Theorem~\ref{ThmY} itself from the weaker version of
Theorem~\ref{ThmZ} corresponding to Theorem~\ref{ThmX}, but, again,
we don't know how to prove the weaker version without proving the stronger.)

\begin{thm}\label{ThmZ}
Fix a small positive $\eps$ and suppose $\gd_x \sim \eps\log n$
for each $x\in V :=[n]$.
Let $M=M(n)\sim (n/r)\log n$ and let $\bH$ be distributed as
$\bH_{n,M}$ conditioned on
\[  
\{d_{\bH}(x)\geq \gd_x ~\forall x\in V\}.
\]  
Then w.h.p.
\beq{PhibH}
\Phi(\bH)> \left[e^{-(r-1)}\log n\right]^{n/r}e^{-o(n)}.
\enq
\end{thm}
\nin
In other words: for $\vs\ll 1$ there is $\varrho\ll 1$ such that if
$M=(1\pm \vs)(n/r)\log n$ and $\gd_x=(1\pm \vs)\eps \log n$ for each $x$,
then
\[
\pr\left(\Phi(\bH)\leq\left[e^{-(r-1)}\log n\right]^{n/r}e^{-\varrho n}\right)
< \varrho.
\]
(It should perhaps be stressed that our argument doesn't work if we allow $\eps=o(1)$;
see the Outline at the end of Section~\ref{Skeleton} and the note following
\eqref{role.of.eps} in Section~\ref{PLC}.)

\mn

In Section~\ref{Skeleton} we derive Theorem~\ref{ThmZ} from
several statements whose proofs will be the main work of this paper.
Outlining that work will be easier once we have
the framework of Section~\ref{Skeleton}, so is postponed until then,
at which point we'll also say a bit about how what we do
here relates to \cite{AsSh}.
We won't assume familiarity with \cite{AsSh}---and will wind up more or less
repeating parts of it---but, as said above, it shows 
the present argument in simpler form, and a reader
of the present work might find it a useful companion.

\mn
\textbf{Usage}

Throughout the paper we fix $r\geq 3$; take $V=[n]:=\{1\dots n\}$,
with $n$ divisible by $r$;
and use $\K$ for $\C{V}{r}$.
We use $v,w,x,y,z$ for vertices and
$\eee,\f,\g, \h,\J$ for $r$-graphs (subsets of $\K$),
or, often, bold versions of these when the $r$-graphs in question are random.
As above, we
abbreviate $\bH^r_{n,M}=\bH_{n,M}$.

We use $d_\h(\cdot)$ and $d_\h(\cdot,\cdot)$ for degree and codegree in $\h$
(thus $d_\h(x)=|\{A\in \h: x\in A\}|$ and
$d_\h(x,y)=|\{A\in \h: x,y\in A\}|$), and 
$\gD_\h$, $\gd_\h$ and $D_\h$ for maximum, minimum and average degrees in $\h$.
We use
$\h_x =\{A\in \h:x\in A\}$ and, for $X\sub V$, $\h[X] =\{A\in \h:A\sub X\}$ 
and $\h-X=\h[V\sm X]$.
We will tend to abusively write $Y\cup a$ and $Y\sm a$ for $Y\cup \{a\}$ and $Y\sm \{a\}$
(in particular $\h\sm A$ for $\h\sm\{A\}$).

For a set $X$ and $p\in [0,1]$, we use $X_p$ for the random subset in which elements
of $X$ appear independently, each with probability $p$.  In all our uses of this
$X$ will be some $\h$; so $\h_p$ is formally in conflict with $\h_x$, but there will never
be any question as to which is meant.
(We will also, beginning with \eqref{Generation}, see $\bH_t$---this always with a 
bold $\bH$, though there would be no confusion in any case.)

As above, we will sometimes use bold for random objects: 
consistently for 
$r$-graphs except when we use $\h_p$, but otherwise 
only when we need to distinguish between a random object and its possible values.

We use mathfrak characters ($\mA,\mB,\mC,\mD, \ldots$)
for properties and events.  
A \emph{property} $\mA$ will usually be a property of $r$-graphs (thus $\mA\sub 2^{\K}$),
and we will say, as convenient, 
``$\h$ has property $\mA$,"
``$\h$ satisfies $\mA$," ``$\h\in \mA$'' or ``$\h\mmodels\mA$."
An \emph{event} is then $\{\bH\mmodels \mA\}$ for some property $\mA$ and random $\bH$,
and may be denoted simply $\mA$ if we have specified $\bH$.
According to what feels natural (or typographically preferable), we use any of the synonymous 
$\mA\wedge\mR$, $\mA\cap\mB$, $\mA\mR$.
As usual, a property of $r$-graphs on $V$ is \emph{increasing} if it cannot be destroyed
by addition of edges.

We assume (as in Theorem~\ref{ThmZ}) that $\eps$ is fairly small.
We will always assume $n$ is large enough to
support our assertions and, following a common abuse,
pretend large numbers are integers.

Asymptotic notation is interpreted as $n \ra \infty$.
We use $a\ll b$ and $a=o(b)$ interchangeably and, similarly,
$a\less b$ is the same as $a<(1+o(1))b$.
We use both ``a.e." and ``a.a." to mean ``for all but a $o(1)$-fraction."
We use $\log$ for natural logarithm and
$a\pm b$ for a quantity within $b$ of $a$.

Where not otherwise stated,
implied constants in $\gO(\cdot)$ and $O(\cdot)$
are allowed to depend on $\eps$.  (Usually they won't, but we will only worry about this 
when it matters.)
A typographical \emph{convention}:  
in exponents only, we will 
use $\epp$ as a substitute for $\gO(1)$; thus different $\epp$'s in a single statement
need not (and will not) be equal.  (For consistency we allow dependence on $\eps$,
but in our uses of $\epp$ this will never make any difference.)

Finally, we set $T=|\K|-M$ 
($M$ as in Theorem~\ref{ThmZ})
and throughout the paper take
\beq{mt}
\mbox{$m_t=|\K|-t \,\, (=\C{n}{r}-t) \,\, \mbox{\emph{and}} \,\, \K^t =\C{\K}{m_t}.$}
\enq
(So $m_{_T}=M$, but we will usually use $m_{_T}$.  
We will always have $t\in [T]$.)
We will often use $m$ for $m_t$ (we think of this as a default, but won't use it 
without notice).
We use $D_m$ for the common value of $D_\h$ for $\h$'s of size $m$
(so when $m=m_t$ and $\h\in \K^t$, $D_\h$ is $D_m$,
\emph{not} the equally plausible $D_t$). 
This may all take a little getting used to, but eventually seemed less annoying than
various alternatives.

\section{Skeleton}\label{Skeleton}

Here we derive Theorem~\ref{ThmZ} from several assertions whose proofs will 
be the main content of the paper.  The discussion here
is similar to that of \cite[Sec.\ 2]{AsSh}.

Recalling that $T=|\K|-M$, we would like to proceed as in 
\cite{JKV,AsSh}, starting from $\bH_0:=\K$ and randomly deleting edges
one at a time to produce the sequence $\bH_0,\bH_1\dots \bH_T$,
with $\bH_T$ the $\bH$ of Theorem~\ref{ThmZ}.
Here uniform deletions will not do, but we may proceed as follows.

Let
\beq{mL}
\mL=\{\J\sub \K: d_\J(x)\geq \gd_x ~\forall x\in V\},
\enq
and to generate $\{\bH_t\}$:
choose $\bH_T$ uniformly from $\mL_{_T}:=\mL\wedge \K^T$;
let $A_1\dots A_T$ be a uniform
ordering of $\K\sm \bH_T$; and for $t\in \{0\dots T\}$ set
\beq{Generation}
\bH_t = \K\sm \{A_1\dots A_t\}
~~(=\bH_T\cup \{A_{t+1}\dots A_T\}).
\enq
(One peculiarity of the present approach is that we \emph{start} with the
object of interest, $\bH_T$, but analyze it as the output of the random sequence
it has been used to generate.)

\mn

The
following two rules governing the law of
$\{\bH_t\}$ are not needed for the present outline but will be the basis for
much of what follows; the easy verifications are left to the reader.
For $\h\sub \K$ of size at least $m_{_T}$,
we use $\bU_\h$ for a uniform $m_{_T}$-subset
of $\h$ and set
\beq{beta}
\gb(\h)=\pr(\bU_\h\in \mL).
\enq
\begin{obs}\label{pr.ww}
Among $\h$'s in $\K^t$,
\[
\pr(\bH_t=\h)~\propto ~\gb(\h).
\]
\end{obs}

\begin{obs}\label{pr.ww'}
Among $A$'s in $\h \in \K^{t-1}$,
\[
\pr(A_t =A|\bH_{t-1}=\h)~\propto ~\gb(\h\sm A).
\]
\end{obs}
\nin

\mn

Set $\Phi(\bH_t) = \Phi_t$ and let $\xi_t$
be the fraction of perfect matchings of $\bH_{t-1}$
that contain $A_t$ (so $\xi_t =\Phi(\bH_{t-1}-A_t)/\Phi_{t-1}$).
Then
\[
\Phi_t = \Phi_0(1-\xi_1) \cdots (1-\xi_t),
\]
or, equivalently,
\beq{mgt1}
\mbox{$\log \Phi_t = \log \Phi_0 +\sum_{i=1}^t\log (1-\xi_i).$}
\enq

\mn

The proof of Theorem~\ref{ThmZ} depends on showing that $\Phi_t$ is likely to stay reasonably
close to its expectation throughout the above evolution.  As will appear,
this is self-reinforcing, with past good behavior favoring
good behavior going forward.

An issue here is that there are 
possibilities for the $\bH_t$'s that don't support our
analysis.  
(The same was true, but in far milder form, in \cite{JKV,AsSh}.)
To deal with this we define (in Section~\ref{SecR}) a collection $\mR$ of
(``reasonable" or ``generic") $\h$'s, write $\mR_t$ for the event $\{\bH_t\in \mR\}$,

and show (mainly in Section~\ref{Foundation})
\beq{Ri}
\pr(\cap_{t\leq T}\mR_t)\ra 1.
\enq
The resulting license to ignore $\bH_t$'s not belonging to $\mR$ will underpin much of
what happens below.

Let 
\beq{gL}
\gL = (r-1)n/r
\enq
and observe that (recall $\log=\ln$)
\beq{mg0r}
\log\Phi_0= \log \frac{n!}{(n/r)! (r!)^{n/r}}
=\frac{n}{r}\log \Cc{n}{r-1}-\gL +O(\log n).
\enq

Now using $m$ for $m_t$ ($=|\K|-t$), set
\beq{gci}
\gc_t=n/(r(m+1)).
\enq
Then $\gc_t$ is
the reciprocal of the average degree in $\bH_{t-1}$, and would be equal to
$\E\xi_t$
if $A_t$ were \emph{uniform} from $\bH_{t-1}$.
That was the situation in \cite{JKV,AsSh},
but here
the $\E\xi_t$'s will require some care; we will show
(recall we are using $\epp$ for a positive constant)
\beq{Exit}
\mbox{if $~\h\in \K^{t-1}\cap\mR\cap\mL,~$ then
$~\E [\xi_t| \bH_{t-1}=\h] < (1+n^{-\epp})\gc_t$}.
\enq
This will mean that, as long as we have not
wandered out of $\mR$, we may think of $\E \xi_t$ as essentially $\gc_t$.
(Of course if $\h\not\in\mL$, the conditioning event is vacuous.)

Let $\mA_t $ be the event
\beq{At}
\left \{\log \Phi_t >\log\Phi_0 - \sum_{i=1}^t\gc_i -o(n)\right\}.
\enq
\nin
(We note, probably unnecessarily, that \eqref{At} refers to some specific $o(n)$,
so that it makes sense to talk about $\mA_t$ for a particular $n$ 
Related points will be common below, and, somewhat departing from common
practice, we will elaborate in a couple places where doing so
seems possibly helpful.)

Noting that
\begin{equation}\label{Em}
\sum_{i=1}^t\gc_i=
\frac{n}{r} \log \left[\C{n}{r}/m\right] +o(1),
\end{equation}
provided $m\gg n$,
and recalling the expression for $\Phi_0$ in \eqref{mg0r}, we find that $\mA_T$ says
\beq{logPhi}
\log \Phi_T > (n/r)\log (rm_{_T}/n)-\gL  -o(n),
\enq
which is the same as \eqref{PhibH};
so Theorem~\ref{ThmZ} is 
\beq{Main1}
\pr(\ov{\mA}_T) =  o(1).
\enq
(We will in fact show
$\pr(\cup_{t\leq T}\ov{\mA}_t) =  o(1)$; see \eqref{3terms}.)

\mn

For \eqref{Main1}
we use the method of martingales with bounded differences.
Here it is  natural---though we will need a variant of this---to consider the martingale
\[
\mbox{$\{X_t =\sum_{i=1}^t (\xi_i -\E[\xi_i|A_1\dots A_{i-1}])\}$}
\]
with associated
difference sequence
\[
\{Z_i = \xi_i -\E[\xi_i|A_1\dots A_{i-1}]\}.
\]

In general, proving concentration for such $X_t$'s
depends on maintaining some
control over the $|Z_i|$'s,
to which end we track, in addition to the $\mR_t$'s,
a second sequence of
events $\mB_t$.
These will be defined in
Section~\ref{BandR}; roughly $\mB_t$ says that no edge of $\bH_t$
is in too much more than its natural share of
perfect matchings.

For $t\leq T $
it will follow trivially from $\mB_{t-1}$
(see \eqref{proofofxibd}) that
\beq{xibd}
\xi_t =O(\gc_t).
\enq
This is more than enough
for the desired concentration but can occasionally fail, since $\mB_{t-1}$ may fail.
To allow for this, as well as possible failures of the $\mR_j$'s,
we slightly modify the above
$X$'s and $Z$'s, setting
\begin{equation}\label{Zi}
Z_i =\left\{\begin{array}{ll}
\xi_i -\E[\xi_i|A_1\dots A_{i-1}] & \mbox{if $\mB_j\mR_j$ holds for all $j<i$,}\\
0&\mbox{otherwise}
\end{array}\right.
\end{equation}
(and
$X_t = \sum_{i=1}^tZ_i$).
As shown in Section \ref{Mart},
a martingale analysis along the lines of
Azuma's Inequality then gives
\beq{conc}
\pr( X_t > \gl  ) < n^{-\omega(1)} ~~~\mbox{for $\gl\gg \sqrt{n}$.}
\enq

We next observe that if $\mB_i\mR_i$ holds for $i<t\leq T$---so
\[
\mbox{$X_t=\sum_{i=1}^t(\xi_i- \E[\xi_i|A_1\dots A_{i-1}])$}
\]
---and
$|X_t| <\sqrt{n}\log n$
(say; there is plenty of room here),
then we have $\mA_t$.
For with these assumptions we have
(using 
\eqref{Exit} and $\sum \gc_i=O(n\log n)$, the latter from \eqref{Em},
and with sums over $i\in [t]$)
\begin{eqnarray}
\mbox{$\sum\xi_i $}&=& \mbox{$X_t + \sum \E[\xi_i|A_1\dots A_{i-1}]$}\nonumber\\
&\leq & \mbox{$X_t +(1+n^{-\epp})\sum \gc_i     ~=~\sum\gc_i +n^{1-\epp} $}\label{sumxi}
\end{eqnarray}
(recall the $\epp$'s needn't agree);
while $\xi_i=O(\gc_i)$ for $i\leq t$ (see \eqref{xibd} and recall we have $\mB_{i-1}$) gives
(using \eqref{gci})
\begin{eqnarray}\label{gc2calc}
\mbox{$\sum_{i=1}^t\xi_i^2$} &=&\mbox{$O(\sum_{i=1}^t \gc_i^2)$} \nonumber\\
&<&
\mbox{$O((n/r)^2\sum\{j^{-2}:j> m_{_T}\}) = O(n/\log n).$}
\end{eqnarray}
Thus, using \eqref{mgt1} (and $\xi_i=o(1)$, as follows from \eqref{gci} and \eqref{xibd}), we have
\[
\log \Phi_t
> \log\Phi_0 -\sum
(\xi_i +\xi_i^2 ) \\
> \log\Phi_0 -\sum
\gc_i  - O(n/\log n)
\]
(where
the $O(n/\log n)$ absorbs the smaller error in \eqref{sumxi}).

\mn

Thus
the first failure, if any, of an $\mA_t$
must occur either because $X_t$ is too large
or because $\mB_i\mR_i$ fails for some $i<t$;
formally, we have
\beq{3terms}
\pr(\cup_{i\leq t}\ov{\mA}_t) < \pr(\cup_{i< t}\ov{\mR}_i)  +
\sum_{i< t}\pr(\mA_i\mR_i\ov{\mB}_i)
+
\sum_{i\leq t}\pr((\cap_{j<i}\mB_j\mR_j)\cap \ov{\mA}_i).
\enq
Here we have already promised in \eqref{Ri} that the first term is $o(1)$;
the last is $n^{-\omega(1)}$ by
\eqref{conc} and the discussion following it; and we will show---this is the main point---
\beq{Bi}
\mbox{for $i<T$,
$ ~\pr(\mA_i\mR_i\ov{\mB}_i) = n^{-\go(1)}.$}
\enq
Thus the l.h.s.\ of \eqref{3terms} is $o(1)$, which in particular gives \eqref{Main1} and,
as already discussed, Theorem~\ref{ThmZ}.  
\qed

\mn
\emph{Outline.}
The structure of our central argument is described in the easy Section~\ref{More},
and an early look at that, with the (even easier) Section~\ref{BandR}, might be helpful.
Here we briefly list contents of the sections
and then say a little about the comparison with \cite{AsSh}.

After recalling a few large deviation facts,
Section~\ref{Mart}
records what we need in the way of martingale concentration, in
a form convenient for a second application in Section~\ref{PLF1}, and
gives the calculation for \eqref{conc}.
Section~\ref{Comparing} develops some reasonably simple machinery
for dealing with $\gb(\h)$'s, the main point being the comparisons of Lemma~\ref{ZZ'lemma}.
Section~\ref{SecR} introduces the rather long list of 
requirements for the property $\mR$, with support for 
\eqref{Ri}
mostly postponed to Section~\ref{Foundation}.  
Section~\ref{Expectations} proves \eqref{Exit}, a first application of Lemma~\ref{ZZ'lemma}.
Section~\ref{BandR} finally defines the central property $\mB$, slightly
reformulates \eqref{Bi} (as \eqref{Bi*}), and disposes of the trivial \eqref{xibd}.
The next five sections are then devoted to the proof of \eqref{Bi*} (which,
as noted following \eqref{Bi}, completes the proof of Theorem~\ref{ThmZ}), as follows.

Section~\ref{More} introduces a few auxiliary properties, with assertions 
concerning them---Lemmas~\ref{LemmaE}-\ref{RCEFB}---that together 
easily imply \eqref{Bi*}.
Lemmas~\ref{LemmaE} and \ref{RCEFB} are from \cite{AsSh} and are just quoted here;
the latter is easy, but the entropy-based Lemma~\ref{LemmaE}
was a key ingredient in the earlier paper
(and is again here), being a first improvement on \cite{JKV} that opens the door to the rest.
Lemma~\ref{LemmaC} is proved in Section~\ref{PLC}.
(It is here that the $\gO(\log n)$ 
lower bound on degrees provided by $\mL$ becomes crucial; see following \eqref{role.of.eps}.)
Lemma~\ref{LemmaF}, which may be considered the core of the whole business,
is proved in Sections~\ref{PLF1}-\ref{PLF3}, with Section~\ref{PLF1}
mostly setting 
out what needs to be done and Sections~\ref{PLF2}-\ref{PLF3} doing it.

Finally, as mentioned above, Section~\ref{Foundation} is concerned with justifying
\eqref{Ri}, with Sections~\ref{Configs}-\ref{Degrees} mainly
developing machinery and Section~\ref{AppR} appying it.
(Sections~\ref{Configs}-\ref{Degrees} are largely self-contained and
maybe amusing in themselves.)

\begin{center}
\emph{All that in idea seemed simple became in practice immediately complex; 
as the waves shape themselves symmetrically from the cliff top, but to the swimmer
among them are divided by steep gulfs, and foaming crests.}
\end{center}
\hspace{3.5in}
Virginia Woolf, \emph{To the Lighthouse}

\mn

The basic difference between this paper and \cite{AsSh} is that 
$\bH_t$ is now chosen according to Observation~\ref{pr.ww}, rather than
uniformly from $\K^t$; in a sense all we are doing here 
is dealing with difficulties occasioned by that change.
This has to date proved harder than one might wish (and maybe harder 
than it needs to be, given that there is quite a lot of room in some of the arguments).
We briefly summarize similarities and differences.

As already mentioned, the above sketch is similar to the one in \cite[Sec.\ 2]{AsSh}.
The easy Section~\ref{Mart} is nearly the same as \cite[Sec.\ 3]{AsSh}
(and gets to skip a couple proofs given there).
Sections~\ref{Comparing} and \ref{Expectations}
have no counterparts in \cite{AsSh}; and 
Sections~\ref{SecR} and \ref{Foundation} have only a faint, routine echo 
in the parts of \cite{AsSh} 
(Section~5 and the appendix) that deal with 
the present $\mR^0$ (there called $\mR$).

The (central) parts of the argument in the remaining sections have rough parallels
in \cite{AsSh}, Sections~\ref{BandR}-\ref{PLC} (here) corresponding to 
Sections~5-7 (there) and Sections~\ref{PLF1}-\ref{PLF3} to Section~9.
(Sections~4 and 8 of \cite{AsSh} prove the present
Lemma~\ref{LemmaE} and correspond to nothing here.)
The biggest changes are in Sections~\ref{PLC}-\ref{PLF2}; 
it is here that we see most clearly the difference between handling
the present $\bH_t$'s and those of \cite{AsSh}, which are ordinary $\bH_{n,m}$'s.
An additional complication is that we must also deal with ($\bH-Z$)'s
(with $Z\in \K$), which in our setting---unlike in \cite{AsSh} where they are again
$\bH_{n,m}$'s---are
different from, and trickier than, the already fairly nasty
$\bH_t$'s.
See e.g.\ the proof of Lemma~\ref{LemmaF}, in particular the parallel setup 
at the beginning of Section~\ref{PLF1},
and then the arguments of Section~\ref{PLF2}, which think mainly of 
$\bG=\bH-Z$ 
and simplify considerably when $\bG=\bH$. 
(The present Section~\ref{PLF3} is just far enough from the corresponding
portion of \cite[Sec.\ 9]{AsSh} that it seems necessary to repeat.)

\section{Concentration}\label{Mart}

Recall that a r.v.\ $\xi$ is \emph{hypergeometric} if, for some
$s,a$ and $k$,
it is distributed as $|X\cap A|$, where $A$ is a fixed $a$-subset
of the $s$-set $S$ and $X$ is uniform from $\C{S}{k}$.
For the standard bounds in Theorem~\ref{T2.1}, see e.g.\
\cite[Theorems~2.1 and 2.10]{JLR}.
\begin{thm}
\label{T2.1}
If $\xi $ is binomial or hypergeometric with  $\mathbb{E} \xi  = \mu $, then for $t \geq 0$,
\begin{align}
\Pr(\xi  \geq \mu + t) &\leq
\exp\left[-\mu\varphi(t/\mu)\right] \leq
\exp\left[-t^2/(2(\mu+t/3))\right], \label{eq:ChernoffUpper}\\
\Pr(\xi  \leq \mu - t) &\leq
\exp[-\mu\varphi(-t/\mu)] \leq
\exp[-t^2/(2\mu)],\label{eq:ChernoffLower}
\end{align}
where $\varphi(x) = (1+x)\log(1+x)-x$
for $ x > -1$ and $ \varphi(-1)=1$.
\end{thm}
\nin
For larger deviations the following consequence of the finer bound in \eqref{eq:ChernoffUpper}
is helpful.
\begin{thm}
\label{Cher'}
For $\xi $ and $\mu$ as in Theorem~\ref{T2.1} and any $K$,
\begin{eqnarray*}
\Pr(\xi  > K\mu) < \exp[-K\mu \log (K/e)].
\end{eqnarray*}
\end{thm}

The next result, proved in \cite{GLSS}, will save us some trouble
at one point
(see Lemma~\ref{Lanemic}).
Say the $\{0,1\}$-valued
r.v.'s $\gz_1\dots \gz_t$ are a \emph{read-k family}
if there are independent r.v.'s $\psi_1\dots \psi_s$ such that each
$\gz_j$ is a function of $(\psi_i:i\in S_j)$ (for some $S_j\sub [s]$)
and
$|\{j:i\in S_j\}|\leq k$ for each $i\in [s]$.
(That is, each $\psi_i$ affects at most $k$ of the $\gz_j$'s.)
\begin{thm}\label{GH}
For $\gz_1\dots \gz_t$ a read-k family
with $\sum\E \gz_j\leq t\rho$,
and any $\gl\in [\rho,1]$,
\[
\pr(\gz_1+\cdots +\gz_t>\gl t)< \exp[-D(\gl\|\rho)t/k],
\]
where $D(\gl\|\rho)=\gl\log (\gl/\rho)+(1-\gl)\log [(1-\gl)/(1-\rho)]$.
\end{thm}
\nin
(The statement in \cite{GLSS} assumes $\sum\E\xi_j = t\rho$,
but implies the present version since $D(\gl\|\rho)$ is decreasing in $\rho\leq \gl$.
We will not need the similar lower tail bound.)

\mn

We now turn to martingales and \eqref{conc}.
Lemma~\ref{Azish} and Proposition~\ref{EeZ}, which we use here and
again in Section~\ref{PLF3}, are Lemma~3.3 and Proposition~3.4 of
\cite{AsSh} and will not be reproved here.
The argument for~\eqref{conc} is really the same as that for (17) in
\cite{AsSh}, but is superficially different enough that it seems best to repeat it.

\mn
\begin{lemma}\label{Azish}
If $Z_1\dots Z_t$ is a martingale difference sequence with respect to the
random sequence $Y_1\dots Y_t$
(that is, $Z_i$ is a function of $Y_1\dots Y_i$ and $\E[Z_i|Y_1\dots Y_{i-1}]=0$), then for
$Z=\sum Z_i$ and any $\vt>0$,
\beq{EeeZ}
\mbox{$\E e^{\vt Z} \leq \prod_{i=1}^t \max \E[e^{\vt Z_i}|y_1\dots y_{i-1}]$}
\enq
and, consequently,
for any $\gl>0$,
\beq{Ebound}
\mbox{$\pr(Z>\gl) <e^{-\vt \gl} \prod_{i=1}^t \max \E[e^{\vt Z_i}|y_1\dots y_{i-1}]$}
\enq
(where $y_i$ ranges over possibilities for $Y_i$).
\end{lemma}
\nin

Both here and in Section~\ref{PLF3}, bounds on the factors in \eqref{EeeZ} are given
by the next observation.  
\begin{prop}\label{EeZ}
For a r.v.\ $W \in [0,b]$ with $\E W \leq a$ and $\vt\in (0, (2b)^{-1}]$,
\beq{Ethetas}
\max\{\E e^{\vt (W -\E W )},\E e^{-\vt (W -\E \theta )}\}
\leq e^{\vt^2 ab}.
\enq
\end{prop}

\begin{proof}[Proof of \eqref{conc}]

Let $\vs_i=O(\gc_i)$ be the bound on $\xi_i$ in \eqref{xibd}.  We will apply
Lemma~\ref{Azish} with $Y_i=A_i$ and
$Z_i$ as in \eqref{Zi} (so $Z=X_t$),
using Proposition~\ref{EeZ} with $b=\vs_i$ and $a=a_i$ the bound in \eqref{Exit}
(with $t=i$; thus $a = (1+n^{-\epp})\gc_i$)
to bound the factors in \eqref{Ebound} (or \eqref{EeeZ}).
(For relevance of the proposition notice that, conditioned on any particular values
$A_1\dots A_{i-1}$, $Z_i$ is
either identically zero (as happens if $\mB_j\mR_j$ has failed for some $j<i$) or
$Z_i=\xi_i -\E[\xi_i|A_1\dots A_{i-1}]$, where $\xi_i\in [0,\vs_i]$ and \eqref{Exit}
bounds the conditional expectation by $a_i$.)
This combination (i.e.\ of Lemma~\ref{Azish} and Proposition~\ref{EeZ})
gives
\[
\mbox{$\pr(X_t > \gl) < \exp[\vt ^2 \sum_{i=1}^t \vs_i a_i -\vt \gl]$}
\]
for any $\gl>0$, provided, say, $\vt\leq 1 $ ($\leq (2\max \vs_i)^{-1}$).
So with
\[
\mbox{$J=\sum_{i=1}^t \vs_ia_i= O(\sum\gc_i^2) = O(n/\log n)$}
\]
(see \eqref{gc2calc}) and $\vt=\min\{1,\gl/(2J)\}$, we have
\[
\Pr(X_t > \gl)<\left\{
\begin{array}{ll}
\exp[-\gl^2/(4J)]&\mbox{if $\gl\leq 2J$,}\\
\exp[-\gl/2]&\mbox{otherwise,}
\end{array}\right.
\]
and \eqref{conc} follows (using $\gl\gg\sqrt{n}$ and, in the first case, $J=O(n/\log n)$).
\end{proof}

\section{Comparing $\gb$'s}\label{Comparing}

A basic difficulty in the present work is that 
we don't know how to estimate the probabilities $\gb(\h)$.
As a substitute, this section develops
some simple machinery for \emph{comparing} $\gb(\h)$'s that will allow us to say,
roughly, that under reasonable restrictions, small changes in $\h$ don't cause
unmanageable changes in $\gb(\h)$.

In general our lives are simpler if we can work with ``binomial'' surrogates for $\bU_\h$
(recall from \eqref{beta} that this is a uniform $m_{_T}$-subset of $\h$).
Set
\beq{gd}
\gd=\eps\log(3/\eps).
\enq
For perspective note that for $\bU=\bU_\K$ (and its binomial relatives below), 
any $x\in V$ (and large enough $n$),
\beq{deg.dev.bd}
\pr(d_{\bU}(x)<\gd_x) < n^{-1+\gd},
\enq
since the first, more precise part of \eqref{eq:ChernoffLower} 
bounds the l.h.s.\ of \eqref{deg.dev.bd} by
$\exp[-(1-\eps \log (e/\eps) +o(1))\log n]$.
(We note that $\gd$ is not meant to suggest $\gd_x$.)

Let 
\beq{vrr}
\mbox{$\vrr= n^{-(1-\gl)/2}$, with $\gl> \gd$ small and fixed}
\enq
(for concreteness we may take
$\gl=2\gd$, but the actual value barely matters),
and for $\h\sub \K$ with $|\h|=m\geq (1+\vrr)m_{_T}$, define
\beq{XhYh}
\mbox{$\bX_\h = \h_p$, $~\bZ_\h = \h_q$,}
\enq
where $p= (1+\vrr)m_{_T}/m$ and $q=(1-\vrr)m_{_T}/m$.
(Recall that $\h_p$ includes edges of $\h$ independently, with probability $p$.
The restriction on $m$ is needed to keep $p$ below $1$; of course $\bZ_\h$ could 
be defined more generally.)
We will compare $\bU_\h$, $\bX_\h $ and $\bZ_\h$, with
the basic observations in 
Lemma~\ref{comp.lemma.1} and the main point of the section 
Lemma~\ref{ZZ'lemma}.

Set
\beq{xi}
\xi=  \exp[-\vrr^2 m_{_T}/3] \,\,(= n^{-\gO(n^\gl)}).
\enq
This value, which is far smaller than we will really need it to be, is chosen so that
(by Theorem~\ref{T2.1})
\beq{XYxi}
\max\{\pr(|\bZ|>m_{_T}),\pr(|\bX|<m_{_T})\} < \xi.
\enq

\mn

Before turning to the main business of this section we record one crude
observation that will sometimes
allow us to more or less ignore very small values of $m_t$
(recall $\mL$ was defined in \eqref{mL}):

\begin{lemma}\label{Psmallm}
If $\h\in \mL$, $|\h|=m< (1+n^{-(2\gd +\epp)})m_{_T}$,  and
\beq{xdHs}
|\{x:d_\h(x)< 1.5 \eps \log n\}| =O(n^{2\gd}),
\enq
then $\gb(\h)> 1-n^{-\epp}$.  
\end{lemma}
\nin
\begin{proof}
Set $\bU=\bU_\h$ and let $X$ be the set in \eqref{xdHs}.  If $\bU\not\in \mL$ then
$\h\sm \bU$ must contain either
\begin{itemize}
\item[(i)] an edge on some $x\in X$, or

\item[(ii)] 
for some $x\in V\sm X$, at least $1.5\eps\log n-\gd_x> 0.4\eps\log n$ of any
given $1.5\eps\log n$ edges of $\h_x$;
\end{itemize}

\nin
and the probability that one of these occurs is at most 
\[
O(n^{2\gd}\log n\cdot n^{-(2\gd+\epp)}) 
+ n\C{1.5\eps\log n}{0.4\eps\log n}n^{-(2\gd+\epp)0.4\eps\log n}
<n^{-\epp}.
\]
\end{proof}

Call $\J\sub \K$ \emph{anemic} if
\beq{anemic}
\mbox{$\J$ has at least $2rn^{2\gd}$ vertices of degree less than $2\eps \log n$.}
\enq

In Lemma~\ref{comp.lemma.1} and Corollary~\ref{comp.cor},
$\bU,\bX$ and $\bZ$ are $\bU_\g$, $\bX_\g$ and $\bZ_\g$ for some $\g\sub \K$
of size at least $(1+\vr)m_{_T}$,
\and $\mS$ is a property of $r$-graphs.
As in much of this work, there is room in the bounds---also, e.g., in the definition
of anemic---and we aim for (relative) 
simplicity rather than anything like optimality.

\begin{lemma}\label{comp.lemma.1}
{\rm (a)}
If $\mS$ is increasing, 
then
\beq{prYQ}
\pr(\bZ\in \mS)-\xi~\leq ~\pr(\bU\in \mS)~\leq ~(1-\xi)^{-1}\pr(\bX\in \mS).
\enq

\nin
{\rm (b)}
Suppose $\mS$ is increasing and membership of $\J$ in $\mS$ is determined by 
$\{x: d_\J(x)\geq \rho_x\}$, with $\rho :=\max\rho_x \leq 1.5\eps\log n$.  If
\beq{prXA}
\pr(\mbox{$\bX$ anemic})<\eta\pr(\bX\in \mS),
\enq
then
$\pr(\bZ\in\mS)>(1-\eta -n^{-\epp}))\pr(\bX\in\mS)$.

\end{lemma}
\nin
(Note \eqref{prXA}, though we will need to check it whenever we use (b), 
is not much of a requirement.)

\begin{proof}
(a)  Since $\mS$ is increasing, we have
\[
\pr(\bZ\in \mS) \leq \pr(|\bZ|> m_{_T})+\pr(\bU\in \mS)
\]
(yielding the first inequality in \eqref{prYQ}) and
\[
\pr(\bX\in\mS)~\geq ~
\pr(|\bX|\geq m_{_T})\pr(\bU\in\mS)
\]
(yielding the second).

\mn
(b)  We couple $\bX$ and $\bZ$ in the usual way:
$\bZ= \bX_{1-\vs}$, with
$\vs = 1-(1-\vrr)/(1+\vrr) < 2\vrr$.
The desired inequality is then
\beq{ETSXY}
\pr(\bZ\in \mS|\bX\in \mS) >1-\eta-n^{-\epp}.
\enq
From \eqref{prXA} we have
\[
\pr(\mbox{$\bX$ anemic}|\bX\in \mS)~\leq ~
\pr(\mbox{$\bX$ anemic})/\pr(\bX\in \mS)~<~ \eta.
\]
On the other hand, $\{\bX\in \mS,\bZ\not\in \mS\}$ implies existence of an $x$ for which
\beq{existsx}
d_{\bZ}(x)<\rho_x\leq d_{\bX}(x),
\enq
meaning that the passage from $\bX$ to $\bZ$ deletes at least $d_{\bX}(x)-\rho_x+1$
members of $\bX_x$.
But
the probability of deleting at least $k$ edges at $x$ is at most $\C{d_{\bX}(x)}{k}\vs^k$;
so
whenever $\bX$ is non-anemic, the probability that \eqref{existsx} occurs for some $x$ is less than
\[   
2rn^{2\gd}\rho\vs + 
 n\C{2\eps\log n}{0.5\eps\log n}\vs^{0.5\eps \log n}.
\]   
Combining these observations gives \eqref{ETSXY}:
\begin{align*}
\pr(\bZ\in \mS|\bX\in \mS) &\geq
\pr(\mbox{$\bX$ non-anemic}|\bX\in \mS)
\pr(\bZ\in \mS|\bX\in \mS, \mbox{$\bX$ non-anemic}) \\
&>1-\eta - n^{-\epp}.\qedhere
\end{align*}
\end{proof}

In Corollary~\ref{comp.cor} and Lemma~\ref{ZZ'lemma} we use
$a\succ b$ for $a > (1-n^{-\epp})b$. 

\begin{cor}\label{comp.cor}
If $\mS$ is as in Lemma~\ref{comp.lemma.1}(b) and
\beq{PYS}
\pr(\bX\in \mS)> n^{\epp}\max\{\xi,\pr(\mbox{$\bX$ anemic})\},
\enq
then
\beq{PYUX}
\pr(\bZ\in \mS) \succ\pr(\bU\in \mS)  \succ \pr(\bX\in \mS)
\,\, (\geq \pr(\bZ\in \mS)).
\enq
\end{cor}
\nin
(Like \eqref{prXA},
\eqref{PYS} should be considered a minor annoyance.)

\begin{proof}
This is three applications of Lemma~\ref{comp.lemma.1}:
the second part of (a) of the lemma gives
$
\pr(\bX\in \mS) \succ \pr(\bU\in \mS);
$
(b), with the second bound in \eqref{PYS}, gives
$
\pr(\bZ\in \mS) \succ  \pr(\bX\in \mS);
$
and combining this with the first bound in \eqref{PYS}, and using 
the first part of (a), gives
$
\pr(\bU\in \mS) \succ  \pr(\bZ\in \mS).  
$
\end{proof}

\begin{cor}\label{betacor}
$\gb(\K)\more\exp[-n^\gd].$
\end{cor}
\nin
(In fact $\gb(\K)\sim\exp[-n^\gd]$, but we don't need this.)

\begin{proof}
With $\bZ=\bZ_\K$, Theorem~\ref{T2.1} 
gives
\[
\pr(d_{\bZ}(x)<\gd_x) < n^{-1+\gd}
\]
(the calculation, which is essentially the same as that for \eqref{deg.dev.bd},
is valid provided $\vr=o(1)$),
and combining this with Harris' Inequality \cite{Harris} yields
\[
\pr(\bZ\in \mL) >(1- n^{-1+\gd})^n \sim \exp[-n^\gd].
\]
The corollary then follows from Lemma~\ref{comp.lemma.1}(a) (with $\rho_x=\gd_x$
and $\mS=\mL$, the subtracted $\xi$ in \eqref{prYQ} being obviously irrelevant here).
\end{proof}

Say $x$ is \emph{dangerous for $\J$} ($\sub \K$) if 
\beq{dangerous}
d_\J(x)<1.5\eps D_\J
\enq
(recall $D_\J$ is average degree in $\J$).
For the next lemma we: assume $\JJJ,\JJJ'\sub \K$, $\eee=\JJJ\cap\JJJ'$,
\beq{mmT}
|\J|=|\J'|=m\geq (1+\vr)m_{_T},
\enq
\[
\mbox{$\JJJ\sm \eee =\{A_1\dots A_\gs \}~$ and $~\JJJ'\sm \eee =\{B_1\dots B_\gs \}$,}
\]
with
\beq{Aimg}
\mbox{$\{A_1\dots A_\gs \}$ a matching}
\enq
and
\beq{t.not.large}
(1\leq ) ~\gs= n^{o(1)};
\enq
let $\kappa$ be the number of $A_i$'s containing vertices that are
dangerous for $\JJJ$;
and set $W=A_1\cup\cdots\cup A_\gs$, noting that
\beq{Wcounts}
\mbox{$|W|\leq \gs r~$ and $~|\{x\in W: x ~\text{dangerous for} ~\JJJ\}|\leq \kappa r.$}
\enq

\begin{lemma}\label{ZZ'lemma}
With the above setup, assume: $\JJJ'\in \mL$;
\beq{nodang}
\mbox{if $m>2rm_{T}$ then $d_\eee(x)\geq 1.5\eps D_\JJJ ~\forall x\in W$}
\enq
(so $\kappa=0$); 
and 
\beq{betaD}
\gb(\JJJ) > n^{\gO(\gs)} \max\{\xi,\pr(\bX_{\J}~ \mbox{anemic})\},
\enq
where \emph{the implied constant doesn't depend on $\eps$}.
Then
\beq{Z'Z}
\gb(\JJJ')\succ n^{-O(\kappa\eps)}\gb(\JJJ)
\enq
(where, of course, the implied constant is again universal).
\end{lemma}
\nin
\emph{Notes.}
The assumptions in \eqref{t.not.large} 
and \eqref{betaD}, while supporting the lemma,
are much weaker than what we will have when we come to use it;
on the other hand, the $\eps$ in \eqref{Z'Z} will be critical at one point (see \eqref{Pg0}).
Of course the ``$\succ$'' (vs.\ ``$>''$)
in \eqref{Z'Z} is unnecessary
if $\kappa\neq 0$.
(It is only in our first application of Lemma~\ref{ZZ'lemma}---in the proof of \eqref{Exit}
in Section~\ref{Expectations}---that the precision of ``$\succ$'' in \eqref{Z'Z} is needed.)

\begin{proof}
Recalling that $\gb(\h)=\pr(\bU_\h\in\mL)$,
we first observe that it suffices to show \eqref{Z'Z} with $\bZ$ in place of $\bU$; that is,
\beq{THB}
\pr(\bZ_{\JJJ'}\in \mL) \succ n^{-O(\kappa\eps)}\pr(\bZ_{\JJJ}\in \mL).
\enq
To see that this is enough, notice that Corollary~\ref{comp.cor}, with $\mS=\mL$ (and $\rho_x=\gd_x$) 
and
\eqref{PYS} given by \eqref{betaD} and Lemma~\ref{comp.lemma.1}(a)
(the latter to say $\pr(\bX_\J\in \mS)\more \gb(\J)$),
gives
\beq{THBB}
\pr(\bZ_{{\JJJ}}\in \mL)\succ \gb({\JJJ});
\enq
that this with \eqref{THB} and \eqref{betaD} gives
\[  
\pr(\bZ_{{\JJJ}'}\in \mL)
> n^{-O(\kappa\eps)+\gO(\gs)}\xi \,\,\,(> n^{\epp} \xi),  
\]   
which by Lemma~\ref{comp.lemma.1}(a) gives
\beq{BHA}
\gb({\JJJ}')\succ\pr(\bZ_{{\JJJ}'}\in \mL);
\enq
and, finally, that the combination of \eqref{BHA}, \eqref{THB} and \eqref{THBB}
gives \eqref{Z'Z}.

\mn

For the proof of \eqref{THB} set $\bZ_{\JJJ}=\bZ$ and $\bZ_{{\JJJ}'}=\bZ'$.  We may
couple these by taking
$
\bZ'\cap \eee=\bZ\cap \eee,
$
with the remaining decisions (those involving the $A_i$'s and $B_i$'s) made independently,
and show a mild strengthening of \eqref{THB}:
\beq{Y'Y}
\pr(\bZ'\in \mL|\bZ\in \mL)\succ n^{-O(\kappa\eps)}.
\enq
Here Harris' Inequality gives
\begin{eqnarray}
\pr(\bZ'\in \mL|\bZ\in \mL) &=&\pr (d_{\bZ'}(x)\geq \gd_x~\forall x\in W|\bZ\in \mL) \nonumber\\
&\geq &\mbox{$\prod_{x\in W}\pr (d_{\bZ'}(x)\geq \gd_x)$.}\label{prodxW}
\end{eqnarray}
On the other hand, now using ${\JJJ}'\in \mL$ and \eqref{nodang}
(the latter just to say that if the first case in \eqref{dYdx} is not vacuous
then $q$ ($=(1-\vr)m_{_T}/m$) $=\gO(1)$), we have,
for $x\in W$,
\beq{dYdx}
\pr(d_{\bZ'}(x)\geq \gd_x)\geq \left\{\begin{array}{ll}
q^{\gd_x}=n^{-O(\eps)}&\mbox{if $d_\eee(x)< 1.5\eps D_\J-1$,}\\
1-n^{-\gO(\eps)}&\mbox{otherwise.}
\end{array}\right.
\enq
Here the second bound is given by Theorem~\ref{T2.1} since 
$\E d_{\bZ'}(x)\more 1.5\eps\log n$, as follows from
$d_{{\JJJ}'}(x)\geq d_{\eee}(x)\geq  1.5\eps D_{{\JJJ}'}-1$.
Finally, inserting the bounds from \eqref{dYdx} in \eqref{prodxW}, 
noting that \eqref{Aimg} implies that all $x$'s in the first part of \eqref{dYdx} 
are dangerous for $\J$, and using
\eqref{t.not.large} and \eqref{Wcounts}, gives \eqref{Y'Y}.\end{proof}

\section{Generics}\label{SecR}

Here we define the property $\mR$, with
most of the discussion supporting \eqref{Ri}
postponed to Section~\ref{Foundation}.
As we will do elsewhere (see Sections~\ref{BandR}, \ref{More} and \ref{PLF1}),
we give the definition for a general $r$-graph $\h$, 
with $n:=|V(\h)|$ and $m:=|\h|\gg n$,
so average degree
\[    
D_\h= mr/n.
\]   
The event $\mR_t$ of Section~\ref{Skeleton} is then $\{\bH_t\in \mR\}$.
(In line with our ``default,'' we may think of $t= \C{n}{r}-m$, but, 
apart from a pseudo-exception at \eqref{m3mT},
the present definitions don't involve $t$.)

We number the several requirements for 
$\mR$---thus $\mR=\mR^0\cap\cdots\cap \mR^4$
(with, again, $\mR^i_t = \{\bH_t\in \mR^i\}$)---to
allow later pointers to what exactly is being used.
The first of these, $\mR^0$ consists of standardish genericity conditions for degrees; \emph{viz.}
\beq{Rg0}
\mbox{a.a.\ degrees in $\h$ are asymptotic to $D_\h$}
\enq
(formally:  there is $\vs=o(1)$ such that $d_\h(x)=(1\pm \vs)D_\h$ for all but $(1-\vs)n$ vertices $x$);
\beq{Rg1}
\gD_\h =O(D_\h), ~~~\gd_\h =\gO(\eps D_\h);
\enq
and
\beq{Rg2}
\mbox{$\max d_\h(x,y)=o(D_\h)$}
\enq
(the max over distinct $x,y\in V$).

Note that $\mR^0$ is robust in that, for any fixed $C$,
\beq{robust}
\mbox{if $\h$ satisfies $\mR^0$ then so does $\h-Z$ for each $Z\sub V$ with $|Z|\leq C$.}
\enq
(Note---though it doesn't matter---this refers to a slightly changed $n$ and $m$.)
We omit the easy justification, just noting that \eqref{Rg1} implies
$D_{\h-Z}\sim D_\h$ and that \eqref{Rg0} and (the lower bound in) \eqref{Rg1}
for $\h-Z$ depend on having \eqref{Rg2} for $\h$. 
We will only use \eqref{robust} with $Z$ a member of $\K$ or the union of two 
disjoint members.

\mn

For the next item, $m_{_T}$ is the $M=M(n)$ of Theorem~\ref{ThmZ}
(so the definition makes sense for a general $\h$, though we will use it only with $\h=\bH_t$).
Say $\h$ is in $\mR^1$ provided that
\beq{m3mT}
\mbox{if $|\h|> 2rm_{_T}$ then $\gd_\h > 2\eps D_\h $}.
\enq
(So---irrelevantly---smaller $\h$ are automatically in $\mR^1$.
This item is minor but will sometimes allow us
to focus on $|\h|$ closer to $m_{_T}$, which is where most of the interest lies.
Of course for $\h=\bH_t$ \eqref{m3mT} implies the second part of \eqref{Rg1},
since $\bH_t\in\mL$.)

\mn

That 
\beq{Ri'}
\pr(\cap_{t\leq T}(\mR^0_t\cap \mR^1_t))\ra 1
\enq
is shown in Section~\ref{AppR}.

\mn

For $Z\sub V$ (and $\h\sub \K$), set
\beq{mKZH}
\mJ_Z(\h)=\{\h'\sub\K:|\h'|=|\h|,\h'-Z=\h-Z\},
\enq
\beq{vpZh}
\vp_Z(\h)=\sum\{\gb(\h'):\h'\in\mJ_Z(\h)\}
\enq
and
\beq{barvpZh}
\barvp_Z(\h)=|\mJ_Z(\h)|^{-1}\vp_Z(\h)
\enq
(the average of $\gb(\h')$ over $\h'$ of size $|\h|$ agreeing with $\h$ off $Z$;
note each of the items in \eqref{mKZH}-\eqref{barvpZh}
is determined by $|\h|$ and $\h-Z$).
Let
\beq{mR1}
\mR^2=\{\h:   \barvp_Z(\h)> n^{-(2r+1)}\gb(\K)\,\,\,\forall Z\in \{\0\}\cup \K\},
\enq
noting in particular (taking $Z=\0$ and recalling Corollary~\ref{betacor}) that
\beq{betaH}
\h\in \mR^2 ~~\Ra ~~ \gb(\h)> n^{-(2r+1)}\gb(\K) > \exp[-(1+o(1))n^{\gd}].
\enq
In Section~\ref{AppR} we will show that for any $Z\sub V$ (and any $t\in [T]$ and $\eta>0$),
\beq{PvpZHT}
\pr(\barvp_Z(\bH_t)<\eta \gb(\K))< \eta,
\enq
whence
\beq{R1bd}
\pr(\exists t ~~\bH_t\not\in \mR^2) < n^{-1}  ~ (=o(1)).
\enq

\mn

With $\ga = m_{_T}/m$ and, as in \eqref{anemic},
$\J$ \emph{anemic} if $d_\J(v) <2\eps \log n$ for
at least $2rn^{2\gd}$ vertices $v$, let
\[   
\mR^3 =\{\h: \pr(\mbox{$\h_\ga$ anemic}) < \exp[-(1-o(1))n^{2\gd}]\}.
\]   
Again in Section~\ref{AppR}, we will show
\beq{prR3}
\pr(\bH_t\not\in \mR^3) < \exp[-(1-o(1))n^{2\gd}].
\enq
We will usually use membership in $\mR^3$ in combination with the next little point.
\begin{obs}\label{mR1Cor}
For fixed $c>0$, $\h$ of size asymptotic to $m$ and $\gz\sim m_{_T}/m$,
if 
\beq{hgzanemic}
\pr(\mbox{$\h_\gz$ anemic}) < \exp[-\gO(n^{2\gd})],
\enq
then
\beq{mR1cor}
|\{x:d_\h(x) < (2-c)\eps D_\h\}|< 2rn^{2\gd}.
\enq
\end{obs}
\nin
In particular, \eqref{mR1cor} holds when $\h\in \mR^3$.
\begin{proof}
If $W$ is a $(2rn^{2\gd})$-subset of the set in \eqref{mR1cor}, then
Harris' Inequality (with Theorem~\ref{T2.1}) gives
\[
\pr(\mbox{$\h_\gz$ anemic})\geq
\pr(d_{\h_\gz}(x)< 2\eps\log n ~\forall x\in W) > (1-o(1))^{|W|}=\exp[-o(n^{2\gd})],
\]
contradicting \eqref{hgzanemic}.
\end{proof}

For $\mR^4$ we use a parameter $\go$ that tends to infinity \emph{slowly}; 
precisely, we want
\beq{go}
1\ll \go\ll \nu,
\enq 
where conditions on $\nu$ are included in the parameter requirements of
\eqref{gznu}-\eqref{param4}.
Those 
requirements could instead 
have been given here, but really belong in
Section~\ref{PLF3}, where they become relevant.
As explained there, they are functions of
some $\gc=o(1)$ that in turn
depends on the ``quality'' of $\mA$ and $\mR^0$ (e.g.\ the speed of the $o(1)$ in $\mA$),
meaning the most we can ask of $\nu$ is that it tend to infinity; this leaves room
for \eqref{go}, but no more.  (We could also make $\go$ a suitably large constant,
but this feels less natural.)

For $Z\sub V$, let
\beq{peex}
\mD_Z=\{\h\in \K^T:\mbox{$\sum_{y\in V\sm Z}(\gd_y-d_{\h-Z}(y))^+> \go$}\},
\enq
and let $\mD = \cup_{Z\in \K}\mD_Z$ (and $\mD_x=\mD_{\{x\}}$).  
It is shown in Lemma~\ref{LB} (see \eqref{13.9b}) that
\beq{PhTPx}
\pr(\bH_T\in \mD)< n^{-\eta\go}
\enq
for some fixed $\eta>0$
(not depending on $\eps$, though we don't need this).

For $t\in [T]$ and $Z\in \K$,
say $\h\in\K^t$ is in $\mR^4(Z)$ if 
\beq{HTDZ}
\pr(\bH_T\in \mD_Z| \bH_t-Z= \h-Z) < n^{-\eta\go/2},
\enq
and set $\mR^4=\cap_{Z\in\K}\mR^4(Z)$.  Then for any $Z$ ($\in \K$),
\begin{eqnarray}
n^{-\eta\go}&>&\pr(\bH_T\in\mD)\nonumber\\
&\geq &\pr(\bH_t\not\in \mR^4(Z))\pr(\bH_T\in \mD_Z|\bH_t\not\in \mR^4(Z))\nonumber\\
&\geq&\pr(\bH_t\not\in \mR^4(Z))n^{-\eta\go/2}\label{n1-o1},
\end{eqnarray}
implying $\pr(\bH_t\not\in \mR^4(Z))<n^{-\eta\go/2}$ and
\beq{prR2}
\pr(\exists t ~\bH_t\not\in \mR^4) < n^{2r}n^{-\go/2} = n^{-\gO(\go)} ~(=o(1)).
\enq
(To make sense of \eqref{n1-o1}
notice that membership of $\bH_t$ in $\mR^4(Z)$ is decided by
$\bH_t-Z$, and the conditioning says precisely that \eqref{HTDZ} does not hold.)

\mn

To recap:  \eqref{Ri} will follow from
\eqref{Ri'}, \eqref{PvpZHT} (which implies \eqref{R1bd}), \eqref{prR3} and \eqref{PhTPx}
(which implies \eqref{prR2});  as noted above, these are all shown in Section~\ref{Foundation}.

\section{Expectations}\label{Expectations}

In this section we take $m=m_t$ and 
prove \eqref{Exit}; recall this said 
(with $\gc_t=n/(r(m+1))$; see \eqref{gci})
\beq{Exit'}
\h\in \K^{t-1}\cap\mR\cap \mL~ \Ra
~\E [\xi_t| \bH_{t-1}=\h] < (1+n^{-\epp})\gc_t.
\enq

Given $\h$ as in \eqref{Exit'}, set,  for $A\in \h$, 
\[
p_A=p_\h(A)=\pr(A_t=A|\bH_{t-1}=\h) \,\,\,(\propto \gb(\h\sm A); ~\mbox{see Observation~\ref{pr.ww'}})
\]
and
\[\xi_A =\xi_\h(A)=\Phi(\h- A)/\Phi(\h).
\]
Then $\xi_A$ is the fraction of perfect matchings of $\h$ that contain $A$, and
\beq{Exim}
\mbox{$\E[\xi_t|\bH_{t-1}=\h] = \sum p_A\xi_A.$}
\enq
We will show that no $p_A$ is much more than the average;
concretely,
\beq{phA}
p_A < (1+n^{-\epp})/(m+1) \,\,\,\forall A\in \h.
\enq
Since $(m+1)^{-1}\sum_{A\in \h}\xi_A =\gc_t~$
(equivalently, $\sum\Phi(\h-A) = (n/r)\Phi(\h)$),
this bounds the r.h.s.\ of \eqref{Exim} by $(1+n^{-\epp})\gc_t$, as desired.

\begin{proof}[Proof of \eqref{phA}]

Set
\[
\h^0= \{B\in \h: d_\h(x)\geq 1.5\eps D_\h+1\,\,\,\forall x\in B\}
\]
and notice that, by Observation~\ref{mR1Cor} and \eqref{m3mT}
(so we use $\h\in \mR^3\cap\mR^1$),
\[     
|\h^0|\left\{
\begin{array}{ll}
\geq m+1- 3\eps r n^{2\gd}D_\h > (1-n^{-\epp})(m+1)&\mbox{if $m\le 2rm_{_T}$},\\
= m+1&\mbox{otherwise.}
\end{array}\right.
\]    
This implies that for \eqref{phA} it is enough to show
\beq{PHB1}
\gb(\h\sm A)< (1+n^{-\epp})\gb(\h\sm B) \,\,\,\forall B\in \h^0,
\enq
since then
\[
p_A = \frac{\gb(\h\sm A)}{\sum_{B\in \h}\gb(\h\sm B)}
~ \leq~ \frac{\gb(\h\sm A)}{\sum_{B\in \h^0}\gb(\h\sm B)}
~< ~(1+n^{-\epp})/(m+1).
\]

For $m< (1+n^{-3\gd})m_{_T}$ (say), \eqref{PHB1}
is give by Lemma~\ref{Psmallm}, according to which 
$\gb(\h\sm B) > 1-n^{-\epp}$ for any $B\in \h^0$ 
(the lemma's hypotheses, $\h\sm B\in \mL$ and \eqref{xdHs},
following from $B\in \h^0$ (with $\h\in \mL$)
and the combination of $\h\in \mR^3$ and Observation~\ref{mR1Cor} respectively).

\mn

For larger $m$ we show \eqref{PHB1} assuming \eqref{phA} fails (which
suffices for our purposes).
We have
\[
\frac{1}{m+1}\sum_{B\in \h}\gb(\h\sm B)~=~\gb (\h) ~>~\exp[-(1+o(1))n^\gd].
\]
[The inequality holds since $\h\in \mR^2$ (see \eqref{betaH}), and 
for the equality we have, with sums over $B\in \h$ (and $\mL_{_T}=\mL\cap \K^T$),
\begin{eqnarray*}
\sum \gb(\h\sm B)&=& \Cc{m}{m_{_T}}^{-1} \sum |\{\U\in \mL_{_T}: \U\sub \h\sm B\}|\\
&=&
\Cc{m}{m_{_T}}^{-1} (m+1-m_{_T})|\{\U\in \mL_{_T}: \U\sub \h\}| = (m+1)\gb(\h).]
\end{eqnarray*}
We thus have
\beq{PHB}
p_A = \frac{\gb(\h\sm A)}{\sum_{B\in \h}\gb(\h\sm B)}
< \frac{1}{m+1}\frac{\gb(\h\sm A)}{\exp[-(1+o(1))n^\gd]}
\enq
and may assume
\beq{BHB}
\gb(\h\sm A) > \exp[-(1+o(1))n^\gd],
\enq
since otherwise \eqref{PHB} implies \eqref{phA}.

Then for \eqref{PHB1} we apply Lemma~\ref{ZZ'lemma} with
${\JJJ}=\h\sm A$ and ${\JJJ}'= \h\sm B$
(so $\eee=\h\sm\{A,B\}$, $\gs=1$ and, unfortunately, $A_1=B$ and $B_1=A$).
Here $B\in \h^0$ implies $\kappa=0$, so the lemma, if applicable, does give \eqref{PHB1};
but its first hypothesis, \eqref{nodang}, holds because $B\in \h^0$, and its second,
\eqref{betaD}, is a weak consequence of \eqref{BHB}
and $\h\in \mR^3$, which implies
$\pr(\mbox{$\bX_\h$ anemic}) <\exp[-(1-o(1))n^{2\gd}]$.
(Note Lemma~\ref{ZZ'lemma} also assumes the lower bound in \eqref{mmT}, but
here we have the stronger $m\geq (1+n^{-3\gd})m_{_T}$.)
\end{proof}

\section{Properties $\mA$ and $\mB$ }\label{BandR}

As in Section~\ref{SecR}, properties in this section, as well as 
Sections~\ref{More} and \ref{PLF1},
are defined for a general $r$-graph $\h$,
and $\mS_t$ is the event $\{\mbox{$\bH_t\in \mS$}\}$.
Here and in Section~\ref{More}---but not quite in Section~\ref{PLF1}---we
again use $n$ and $m$ 
for the numbers of vertices and edges of $\h$, and $\K$ for $\C{V(\h)}{r}$.

For the remainder of the paper we will 
tend to use $A$ for edges and $Z$ or $U$ for general $r$-sets.
We assume throughout that we have fixed some positive $\eps$
(it will be essentially the one in Theorem~\ref{ThmZ}), upon which
the implied constants in ``$O(\cdot)$" and ``$\gO(\cdot)$" may depend.

\mn

We say $\h$ has the property $\mA$ 
if
\beq{Ag}
\log \Phi(\h) >\log\Phi_0 - \frac{n}{r} \log \left[\C{n}{r}/m\right] -o(n).
\enq
(Recall from \eqref{Em} that for $m=m_t$,
the main subtracted term here is essentially $\sum_{i=1}^t\gc_i$;
so, as promised, $\{\bH_t\in \mA\}$ is the $\mA_t$ of Section~\ref{Skeleton}.)

\mn

For $\mB$ a little notation will be helpful.
For  a finite set $S$ and $\ww: S \rightarrow \Re^+$
($:=[0, \infty)$),
set
\[
\overline \ww (S)= |S|^{-1} \sum_{a\in S} \ww(a),
\]
\[
\max \ww (S)= \max_{a \in S}\ww(a),
\]
and
\[
\maxr \ww(S)= \overline \ww (S)^{-1}\max \ww (S).
\]

For $\h \sub \K$ define $\ww_\h:\K\ra \Re^+$ by
\[
\ww_\h (Z) = \Phi(\h -Z),
\]
and say $\h$ has the property $\mB$ if
\[
\maxr \ww_{\h}  (\h ) =  O(1).
\]
(So $\mB$ says the number of p.m.s containing any particular $A\in \h$ is not too large
compared to the average.  
Note that the implied constant here \emph{does} depend on $\eps$, its natural
value being roughly $1/\eps$:  on average over $A\ni x$, the fraction of p.m.s of 
$\h$ containing $A$ is $1
/d_\h(x)$;
and the $d_\h(x)$'s, while typically around $D_\h$, can be as small as (about) $\eps D_\h$.)

Then $\mB_t$ ($=\{\bH_t\mmodels\mB\}$) is as in Section~\ref{Skeleton}
and \eqref{Bi} is
\beq{Bi*}
\mbox{for $t< T$,
$~~\pr(\bH_t\mmodels \mA\mR\ov{\mB}) =n^{-\go(1)}.$}
\enq
(More formally:  there is a fixed $C$, depending on the $o(\cdot)$'s
and implied constants in $\mA$ and $\mR$, such that
$\pr(\{\bH_t\mmodels\mA\mR\}\wedge\{\maxr\ww_{\bH_t}(\bH_t) > C\})=n^{-\go(1)}$.)

As mentioned at the end of Section~\ref{Skeleton}, \eqref{Bi*}
is shown in Sections~\ref{More}-\ref{PLF3},
and with
\eqref{Ri} (likelihood of the $\mR_t$'s, established in Section~\ref{Foundation})
will complete the proof of Theorem~\ref{ThmZ}.

\mn

We conclude this section with the promised
\beq{proofofxibd}
\mbox{$\mB_{t-1}$ implies
{\rm \eqref{xibd}}.}
\enq
(Recall \eqref{xibd} says $\xi_t=O(\gc_t)$, where, as in \eqref{gci}, 
$\gc_t = n/(r(m+1))$.)
\begin{proof}
If $\bH_{t-1}=\h$, then
$\xi_t\leq\max_{A\in \h} \ww_\h(A)/\Phi(\h)$,
while $\gc_t$ is the average of these ratios, since
\[
\mbox{$\sum_{A\in \h}\ww_\h(A)=\Phi(\h)n/r$}
\]
(and $|\h|=m+1$).
This gives \eqref{proofofxibd}.
\end{proof}

\section{More properties}\label{More}

We will get at $\mB$ 
\emph{via} several auxiliary properties.
We introduce the first three of these here (there will be a couple more in Section~\ref{PLF1}),
together with
assertions concerning them that, as shown below, easily imply \eqref{Bi*}.
As mentioned earlier,
two of these assertions are from \cite{AsSh} and the others (which also have counterparts in 
\cite{AsSh}) are proved in the next four sections.

With $n$ and $\K$ again the size and collection of $r$-subsets
of $V(\h)$, the properties of interest here are:

\begin{itemize}

\item
[$\mC$:] $~~$ \emph{if $Z\in \K$ satisfies}
\beq{WZ}
\ww_\h(Z) > \Phi(\h)e^{-o(n)},
\enq
\emph{then for any $x\in Z$,}
\beq{wZxy}
\mbox{$\ww_\h((Z\sm x)\cup y)\more \ww_\h(Z)d(x)/D_\h~$ for a.e.\ $y\in V\sm Z$;}
\enq

\item
[$\mE$:]
$~~~
\ww_\h (A)\sim \Phi(\h)/D_\h~$ \emph{for a.e.}\ $A\in  \h$;

\item
[$\mF$:]
$
~~~\ww_\h(Z)\sim \Phi(\h)/D_\h~$ \emph{for a.e.}\ $Z \in \K$.

\end{itemize}
(More formally, e.g.\ for $\mE$:
there is $\vs=\vs(n)=o(1)$ such that
$|\{A\in \h:  \ww_\h(A)\neq (1\pm \vs)\Phi(\h)/D_\h\}| < \vs|\h|$.)
For perspective on $\mE$ and $\mF$, notice that
\[    
(\overline \ww_\h(\h)=) ~ |\h|^{-1}\sum_{A\in \h}\ww_\h(A)
= |\h|^{-1} \Phi(\h)n/r =\Phi(\h)/D_\h.
\]    

\mn

For the next little bit we use
\[
\mX \Rastar
\mZ
\]
to mean $\pr(\mX\ov{\mZ})=n^{-\go(1)}$;
e.g.\ the probability bound of \eqref{Bi*} is
\beq{B''}
\{\bH_t\models\mA\mR\}\Rastar \{\bH_t\models \mB\}.
\enq

\mn

The aforementioned assertions are as follows.
(In Lemmas~\ref{LemmaF} and \ref{LemmaC} we assume $t\in [T]$.)

\begin{lemma}\label{LemmaE}
If $\h$ satisfies $\mA\mR^0$ then it satisfies $\mE$.
\end{lemma}
\begin{lemma}\label{LemmaF}
With $\wedge_Z$ ranging over $Z$ as in \eqref{WZ},
\[   
\mbox{$\{\bH_t\models \mA\mR\} \Rastar \{\bH_t\models \mF\}
\wedge\bigwedge_Z\{\bH_t-Z\models \mF\}.$}
\]   
\end{lemma}
\nin

\begin{lemma}\label{LemmaC}
For $x\in Z\in \K$,
\beq{htAR}
\{\bH_t\models \mR\}\wedge \{\bH_t-Z\models \mF\}\Rastar \{(\bH_t,Z,x)\models \eqref{wZxy}\}.
\enq
\end{lemma}
\begin{lemma}\label{RCEFB}
If $\h$ satisfies $\mR^0\mF\mC$ then it satisfies $\mB$.
\end{lemma}

Lemmas~\ref{LemmaF}-\ref{RCEFB} 
immediately imply \eqref{Bi*} (in the form \eqref{B''}):
the first two give
\[
\{\bH_t\models \mA\mR\} \Rastar \{\bH_t\models \mF\mC\}
\]
and the story is then completed by Lemma~\ref{RCEFB}.

\mn

\mn
\emph{Remarks.}
The crucial contribution of Lemma~\ref{LemmaE}
is that it 
allows us to replace $\mA\mR$ by $\mA\mR\mE$ in
Lemma~\ref{LemmaF}; more precisely: 
Recall from \eqref{robust}
that $\h\models\mR^0$ implies $\h-Z\models\mR^0$ for every $Z$ $\in \K$, and note that
$\h\models\mA$ easily implies $ \h-Z\models\mA$ for any $Z$ as in \eqref{WZ}.
These observations, 
with Lemma~\ref{LemmaE}, say that 
in proving any of the assertions 
$\{\bH_t\models \mA\mR\} \Rastar \{\bG\models \mF\}$ of
Lemma~\ref{LemmaF}, we may 
we replace the l.h.s.\ by 
\beq{LemmaFeq''}
\{\bH_t\models \mA\mR\} \wedge \{\bG\models \mE\}.
\enq
Lemma~\ref{LemmaF} then embodies the idea that $\mE\ov{\mF}$ is unlikely 
for a random $\bG$ (here either $\bH_t$ or $\bH_t-Z$)
because
the distribution of the $\ww_{\bG}(A)$'s ($A\in \bG$) should reflect
that of the $\ww_{\bG}(Z)$'s ($Z \in \K$).
As said in \cite{AsSh},
we regard this natural point
as the heart of our argument.

The more important part of Lemma~\ref{LemmaF} 
is that involving $\bH_t-Z$, which provides
input for Lemma~\ref{LemmaC}.
Its other use, allowing us to assume $\mF$ in Lemma~\ref{RCEFB},
is convenient but less critical:  with
more effort one can show directly that $\{\bH_t\models\mR\mC\ov\mB\}$ is unlikely.

\mn

The nonprobabilistic Lemmas~\ref{LemmaE} and \ref{RCEFB} are 
Lemmas~6.1 and 6.4 of \cite{AsSh}
(the $\mR$ there being the present $\mR^0$),
and their proofs will not be repeated here.
(As mentioned in Section~\ref{Skeleton}, Lemma~\ref{RCEFB}
 is easy, but Lemma~\ref{LemmaE}
was one of the main points of \cite{AsSh}.)
Lemmas~\ref{LemmaC} and \ref{LemmaF} are proved (in this reverse order)
in Sections~\ref{PLC} and \ref{PLF1}-\ref{PLF3} respectively.

\section{Proof of Lemma~\ref{LemmaC}}\label{PLC}

Fix $x\in Z\in \K$
and let $\bH=\bH_t$, $Y=Z\sm x$ and $W=V\sm Z$.
We now use
$d_x$ for $d_{\bH}(x)$.

Recall (see \eqref{Generation}) that $\bH=\bH_T\cup \bJ$ with $\bH_T$ uniform from $\mL_{_T}$ and
$\bJ $ uniform from the ($T-t$)-subsets of $\K\sm \bH_T$.
This gives a joint distribution on $(\bG,\bG_T,\bN_T,\bH_x\sm \bN_T)$,
where $\bG=\bH-x$, $\bG_T=\bG\cap \bH_T$
and $\bN_T= \bH_T\sm\bG_T$ ($=(\bH_T)_x$).
Given $\bG=\g$
we choose $\bG_T,\bN_T,\bH_x\sm \bN_T$ (in this order),
and set $g=|\bG|$ and $g_{_T}=|\bG_T|$.

The law of $\bG_T$ will not concern us, but the next two observations will be helpful.
First, the law of $\bN_T$ depends only on $\bG_T$:
it is uniform measure on
\beq{NT}
\{\N\sub \K_x:|\N|=m_{_T}-g_{_T},~ \N \cup \bG_T \in \mL_{_T}\}.
\enq
Second, $\bH_x\sm \bN_T$ is chosen uniformly from the $(m-g-|\bN_T|)$-subsets of $\K_x\sm \bN_T$;
in particular its law depends only on $\bN_T$ and $g$.

\mn

We are really interested in the choice of $\bH_x$, for which
we may assume whatever consequences of $\bH\in \mR$ can be read off from $(\bG,\bG_T)$;
of these we will use just $D_{\bH-Z}\sim D_{\bH}$, $d_x=\gO(\eps\log n)$
(both given by \eqref{Rg1}) and $\bH\in \mR^4(Z)$.
We are then also entitled to assume
\beq{GTPx}
\bH_T\not\in \mD_x
\enq
(see \eqref{peex}),
since $\bH\in \mR^4(Z)$ says that the probability that 
$\bH_T\in \mD_Z$ ($\supseteq \mD_x$)---an event
decided by $\bG_T$---is $n^{-\gO(\go)}$
(recall $\go$ was introduced at \eqref{go}).
Armed with these assumptions, we continue.

\mn

Set $\Phi' = \Phi(\bH-Z)/D_{\bH}$ (a function of $\bG$)
and
recall that $\bH-Z\in \mF$
says (using $D_{\bH-Z}\sim D_{\bH}$)
\beq{WGZ*}
\mbox{$\ww_{\bH-Z}(U)\sim \Phi' 
~$
for a.e.\ $U\in \K[W]$.}
\enq
On the other hand, for any $y\in W$,
\beq{WHYz}
\ww_{\bH} (Y\cup y) = \sum\{\ww_{\bH-Z}(S\cup y):S\in \Cc{W\sm y}{r-1}, ~S\cup x\in \bH\},
\enq
and \eqref{WGZ*} implies
\beq{WGz}
\mbox{for a.e. $y\in W$, $~\ww_{\bH-Z}(S\cup y)\sim \Phi' ~$ for a.e.\ $~S\in \Cc{W\sm y}{r-1}$.}
\enq
It is thus enough to show that (under the assumptions in the preceding paragraph) the inequality in
\eqref{wZxy} is unlikely to fail for any $y$ as in \eqref{WGz}.

\bn

We are now choosing $\bH_x=\bN_T\cup (\bH_x\sm \bN_T)$.
Given $y$ as in \eqref{WGz}, let\footnote{A more formal version:
\eqref{WGz} says there is $\vs=o(1)$ so that for all but $\vs n$ $y$'s,
\[
|\{S\in \Cc{W\sm y}{r-1}:\ww_{\bH-Z}(S\cup y)\neq(1\pm \vs)\Phi'\}| < \vs n^{r-1}.
\]
Then ``$y$ as in \eqref{WGz}" is one of these, and 
the condition in $\mJ$ is
$\ww_{\bH-Z}(S\cup y)\neq(1\pm \vs)\Phi' $.}
\[
\mI=\{S\in \Cc{W\sm y}{r-1}:\ww_{_{\bH-Z}}(S\cup y)\not\sim \Phi'\}
\cup \{S\in\Cc{V\sm x}{r-1}: S\cap (Y\cup y)\neq \0\}
\]
and $\tilde{\mI}=\{S\cup x: S\in \mI\}$,
and notice that
\beq{Ismall}
|\tilde{\mI}|=|\mI|= o(n^{r-1}).
\enq
(The first set in the definition of $\mI$ consists of the exceptional $S$'s in \eqref{WGz},
so is of size $o(n^{r-1})$;
the second has size $\Theta(n^{r-2})$ and should be ignored.)

Noting that $d_x$ ($=d_{\bH}(x)$)
is determined by $\bG$, we observe that
the sum in \eqref{WHYz} is at least $(1-o(1))d_x\Phi'$---so the inequality in \eqref{wZxy}
holds---\emph{provided}
\[
|\bH_x\cap \tdI|=o(d_x).
\]
So it is enough to show, for some $\gz=o(1)$,
\beq{maxP}
\max\{\pr(|\bN_T\cap \tdI|> \gz d_x),\pr(|(\bH_x\sm \bN_T)\cap \tdI|> \gz d_x)\} = n^{-\go(1)};
\enq
as we will see, this is true whenever
\beq{vrgg}
\gz\gg \max\{(\log(n^{r-1}/|\mI|))^{-1}, \go/\log n\}.
\enq
(Recall we assume $\go\ra\infty$ \emph{slowly}, so the second bound is small.)

Note that for $\bH_x$ \emph{uniform} from $\C{\K_x}{d_x}$
and $\gz$ satisfying (just) the first bound in \eqref{vrgg},
$\pr(|\bH_x\cap \tdI|> \gz d_x)=n^{-\go(1)}$ 
is given by Theorem~\ref{Cher'}, using \eqref{Ismall} and $d_x=\gO(\eps\log n)$;
namely, since $\E |\bH_x\cap \tdI| \asymp n^{-(r-1)}|\mI|d_x$,
Theorem~\ref{Cher'} bounds the probability in question by
\beq{role.of.eps}
\exp[- \gz\log (\gz n^{r-1}/(e|\mI|))d_x]=\exp[-\go(\log n)].
\enq
(It is this point---more precisely, its analogues below---that collapses if 
Theorem~\ref{ThmZ} allows $\gd_x=o(\log n)$.)
So we are trying to show that the present distribution (see \eqref{NT})
doesn't behave too differently.

For any choice of $\bN_T$, $\bH_x\sm \bN_T$ is a uniform subset of some size less than $d_x$
from a universe of size $\Theta(n^{r-1})$, so, as above,
\[
\pr(|(\bH_x\sm \bN_T)\cap \tdI|> \gz d_x) = n^{-\go(1)}.
\]
While something similar is clearly true for 
$\pr(|\bN_T\cap \tdI|> \gz d_x)$,
I don't see how to say
it's just trivial.
A nice fly-with-a-sledgehammer argument runs as follows.

Given $\bG_T$, let
$u_z =\gd_z-d_{\bG_T}(z)$, 
$
J=\{z\in V\sm x:u_z>0\}
$
and
\beq{usum}
\mbox{$u=\sum_{z\in J}u_z \leq \go$}
\enq
(with the inequality from \eqref{GTPx}; again see \eqref{peex}).
Then $\bN_T$ is distributed as a uniform $(m_{_T}-g_{_T})$-subset, $\bN$, of $\K_x$
conditioned on
\[
\mS:=\{d_{\bN}(z)\geq u_z ~\forall z\in J\}.
\]
Set $\mN=\{|\bN\cap \tdI| > \gz d_x\}$ and $\mN'=\{|\bN\cap \tdI| > \gz d_x-u\}$,
and notice that if $\mS \mN$ holds, then $\mS$ and $\mN'$ 
\emph{occur disjointly} at $\bN$; that is, there are disjoint $\A , \B\sub\bN$
such that $\bN\supseteq \A $ implies $\mS$ and
$\bN\supseteq \B $ implies $\mN'$
(see \cite{BK} or e.g.\ \cite{Grimmett}).
But a beautiful result of van den Berg and Jonasson \cite{BJ} bounds the probability
of this disjoint occurrence by $\pr(\mS)\pr(\mN')$, yielding
\[
(\pr(|\bN_T\cap \tdI|> \gz d_x) =)\,\,\,
\pr(\mN|\mS)\leq \pr(\mN')= n^{-\go(1)},
\]
with the $n^{-\go(1)}$ again given by Theorem~\ref{Cher'} (now 
using $u\leq \go$, the second bound in \eqref{vrgg} and, again, $d_x=\gO(\eps\log n)$ to say the 
subtracted $u$ in $\mN'$ is irrelevant).

\section{Proof of Lemma~\ref{LemmaF}:  setting up}\label{PLF1}

In this and the next two sections, $\bH$ is $\bH_t$, 
we use $m$ for $m_t$ and $D$ for $D_m$,
and $\bG$ is either $\bH$ or $\bH-Z$
(with $Z$ as in \eqref{WZ}), with law in either case denoted $\vp$;
thus
\beq{vp}
\vp(\g):=\vp(\bG=\g)\propto \left\{\begin{array}{ll}
\gb(\g)&\mbox{if $\bG=\bH$,}\\
\sum_{\h- Z=\g}\gb(\h)&\mbox{if $\bG=\bH-Z$},
\end{array}\right.
\enq
where $\g$ ranges over possibilities for $\bG$ and $\h$ over $\K^t$
(or just over possibilities for $\bH$, since the rest don't contribute to \eqref{vp}).
Note that the sum in the second part of \eqref{vp} 
is the common value of $\vp_Z(\h)$ for $\h\in \K^t$
with $\h-Z=\g$; see \eqref{vpZh}.

We will refer to the two parts of \eqref{vp} as Case 1 and Case 2 (respectively).
Note that Case 2 includes Case 1 if we allow $Z=\0$; 
we have included the first part of  \eqref{vp} to emphasize the distinction, 
but from this point (until the end of Section~\ref{PLF3}) 
\emph{$Z$ is either empty or as in \eqref{WZ},}
and we set $V'=V\sm Z$ and $\K'=\K-Z$.  
In what follows we will usually be thinking of the more demanding Case 2;
but the arguments also make sense in Case 1,
where they often simplify, sometimes drastically.
A few comments on 
Case 1 appear in square brackets.

\mn

As observed at \eqref{LemmaFeq''}, Lemma~\ref{LemmaE} says that 
Lemma~\ref{LemmaF} is equivalent to the assertion that for each $\bG$ in 
\eqref{vp},
\beq{LemmaFeq'}
\pr(\bH\models \mA\mR, \bG\models \mE\ov{\mF}) = n^{-\go(1)}.
\enq
Note we may assume here that
\beq{sillyp}
m = |\K|-\gO(|\K|);
\enq
for $\bH\models \mR^0$ implies that $|\bG|\sim m$ for each $\bG$ in \eqref{vp},
so if \eqref{sillyp} fails then for each such $\bG$, $\mE$ and $\mF$ are equivalent 
and \eqref{LemmaFeq'} is vacuous.
(This rather silly point will be needed for \eqref{Ubd''}.)

\mn

It will be convenient to further reformulate as follows.
For any $\g$ set
\[
\eta(\g) =\inf\{\eta:|\{U\in \K: \ww_\g(U)\neq (1\pm \eta)\Phi(\g)/D_\g\}|< \eta |\K|\}.
\]
Then $\{\g\models \mF\} = \{\eta(\g)=o(1)\}$ and \eqref{LemmaFeq'} is equivalent
to\footnote{With $\mG=\{\bH\models \mA\mR, \bG\models \mE\} $ and $\mH(\nu) =\{\eta(\bG)> \nu\}$,
\eqref{LemmaFeq'} says
\[
\mbox{there is $\vs=o(1)$ such that $ \pr(\mG\wedge\mH(\vs))=n^{-\go(1)}$,}
\]
while \eqref{L7.2'}  implies
\[
\forall k,  ~ \pr(\mG\wedge\mH(1/k))<n^{-k}~~\mbox{for $n\geq n_k$};
\]
and we get the former from the latter by taking $\vs(n)  =(\max\{k:n_k\leq n\})^{-1}$.}
\beq{L7.2'}
\mbox{\emph{for any fixed $\theta>0$, 
$~\pr(\bH\mmodels \mA\mR, \bG\mmodels\mE,\eta(\bG)>2\theta) =n^{-\go(1)}$.}}
\enq
(The $2\theta$ will be convenient below.)
So for the rest of this section we fix $\theta>0$ and aim for \eqref{L7.2'}.

Set
\beq{Phi'}
\Phi'=\Phi(\bG)/D_{\bG}.
\enq
Notice that $\{\bG\mmodels \mE\}\wedge \{\eta(\bG)>2\theta\}$
implies
\begin{itemize}

\item[$\mQ$:]
\emph{$\ww_{\bG}(A)\sim \Phi'$ for a.e.\ $A\in  \bG$, but
$\ww_{\bG}(\uu )\neq (1\pm 2\theta) \Phi'$ for
at least a ($2\theta$)-fraction of the $\uu $'s in $ \K'\sm\bG$.}

\end{itemize}

\nin
So it is enough to show
\beq{Qbd}
\pr(\bH\mmodels\mA\mR, \bG\mmodels\mQ) =n^{-\go(1)}.
\enq

For the proof of this we work with
an auxiliary random set $\bT$
chosen uniformly from $\C{\bG}{\tau}$, where $\tau$, which will be specified later
(see the paragraph containing \eqref{gznu}-\eqref{param4}), will at least satisfy
\beq{tau}
\go\log n \ll \tau \ll \log^2n.
\enq
(As explained following \eqref{go}, $\go$ is really chosen 
\emph{after} the parameters of \eqref{gznu}-\eqref{param4}.)
We set $\bF=\bG\sm \bT$ and
\[
\gz= e^{-\tau/D},
\]
and will be interested in a property of the pair $(\bG,\bT)$ (or $(\bF,\bT)$),
\emph{viz.}
\begin{itemize}
\item[$\mV$:]
\emph{$\ww_{\bF}(A)\sim \gz \Phi'$ for a.e.\ $A\in   \bT$, but
$\ww_{\bF}(\uu )\neq (1\pm \theta) \gz \Phi'$ for
at least a $\theta$-fraction of the $\uu $'s in $ \K'\sm\bG$,}
\end{itemize}
Note
$\gz \ww_{\bG}(\uu )$ is a natural approximate value for $\ww_{\bF}(\uu )$, since
each p.m.\ of $\bG-U$ survives in $\bF$ with probability
roughly (actually, asymptotically) $(1-\tau/m)^{n/r}\sim \gz$; \emph{cf.} \eqref{1-vs}.

\mn

We will exploit the familiar leverage derived from the interplay
of two natural ways of generating $(\bG,\bT)$:
\begin{itemize}
\item[(A)]
choose $\bG$ and then $\bT$
(as above);

\item[(B)]
choose $\bF $ and then $\bT$ (determining $\bG=\bF\cup\bT$).

\end{itemize}

\nin
(In Case 2, analysis for (A) will involve choosing
$\bH$ rather than just $\bG$.)

\mn

Notice that, given $\bF$, the law of $\bT$ is given by
\beq{reweightT}
\pr(\bT=\T)\propto
\vp(\bF\cup \T),
\enq
where $\T$ ranges over $\tau$-subsets of $\K'\sm \bF$
(and $\vp$ is as in \eqref{vp}).

We will not need to know much about the law of $\bF$, but will want to restrict attention
to reasonably well-behaved possibilities. 
Thus we will define a property $\mN$ specifying a few desirable
features of $\bF$,
and,
now writing $\mA\mR$ for $\{\bH\mmodels \mA\mR\}$, 
$\mQ$ for $\{\bG\mmodels \mQ\}$, $\mN$ for $\{\bF\mmodels \mN\}$,
and $\mV$ for $\{(\bG,\bT)\mmodels \mV\}$,  show
\beq{U|Q}
\pr(\mV \mN|\mA\mR\mQ) = 1-o(1)
\enq
and
\beq{Ubd}
\pr(\mV\mN) = n^{-\go(1)}
\enq
These give \eqref{Qbd}, since
\[
\pr(\mA\mR\mQ) =\pr(\mA\mR\mQ\mV\mN)/\pr(\mV\mN|\mA\mR\mQ)\leq
\pr(\mV\mN)/\pr(\mV\mN|\mA\mR\mQ).
\]
(So \eqref{U|Q} is more than is really needed here.)

\mn

Noting that we have specified $m$ and $Z$, we take $\mN$
to be the property comprising \eqref{N1}-\eqref{N4} below.
The first of these takes a little preparation.
Set,
for $g\leq m$ (in what follows, $g$ will be $|\bG|$):
\beq{mI}
\mI(g) = \C{\K_Z}{m-g},
\enq
with $\K_Z:=\cup_{x\in Z}\K_x$, and
\beq{mI*}
\mI^*(g)=\{\I\in \mI(g):d_\I(x)\geq \gd_x~\forall x\in Z\};
\enq
for $\g'\in \C{\K'}{g}$,
\[
\barvp(\g') =|\mI(g)|^{-1}\vp(\g')
\]
(with $\vp$ as in \eqref{vp});
and for $\f\sub \K'$ with $|\f| =g-\tau$,
\beq{mH}
\mH(\f)
=\{\g'\in \Cc{\K'}{g}:\g'\supseteq \f\} 
\enq
and
\beq{mH*}
\mH^*(\f) =\{\g'\in \mH(\f):d_{\g'}(x)\geq \gd_x~\forall x\in V'\}.
\enq
Like $\vp$ itself, $\barvp$ should recall $\mR^2$:  
$\barvp(\g')$ is $\barvp_Z(\h)$ for any $\h\in \K^t$ with $\h-Z=\g'$.
[In Case 1, $g=m$, $\mI^*(g) = \mI(g)=\{\0\}$, and $\barvp(\g)=\vp(\g) =\gb(\g)$.]
Note that $\mH(\f)$ includes possibilities for $\bG$ given $\bF=\f$, but 
typically also some (irrelevant) \emph{im}possibilities; 
e.g.\ in Case 1 anything in $\mH(\f)\sm\mH^*(\f)$.

In two places below it will be convenient to first dispose of the easy case of very small $m$, 
allowing us to restrict attention to (say)
\beq{m.not.small}
m > (1+n^{-3\gd})m_{_T}.
\enq

With $g=|\f|+\tau$,
the first requirement for $\f\mmodels \mN$ is 
\beq{N1}
\barvp(\g') \more\left\{
\begin{array}{ll} 
|\mI^*(g)|/|\mI(g)| \,\,\forall \g'\in \mH^*(\f)&\mbox{if $m$ violates \eqref{m.not.small},}
\\
n^{-(2r+1)}\gb(\K) \,\,\forall \g'\in \mH(\f) &\mbox{if $m$ satisfies \eqref{m.not.small}.}
\end{array}\right.
\enq
A \emph{little} perspective:  for \eqref{Ubd} we will use viewpoint (B), bounding
$\pr(\mV|\bF=\f)$ for $\f\in \mN$, and will want to say that the law of $\bG$ under this
conditioning is not too awful; but that law is governed by $\vp$ (again, see \eqref{vp}),
which might at least \emph{suggest} relevance of \eqref{N1}.

The other, more easily stated requirements (for $\f\models \mN$) are 
\beq{N2}
\mbox{if $m>2rm_{_T}$ then $d_{\f}(x)> 1.5\eps D$ $\forall x\in V'$};
\enq
\beq{N3}
\mbox{with $\ga=m_{_T}/m$,
$~\pr(\f_\ga~ \mbox{anemic}) < \exp[-(1-o(1))n^{2\gd}]$}
\enq
(see \eqref{anemic} for ``anemic''); and (with $\go$ as in \eqref{go} and \eqref{tau})
\beq{N4}
\mbox{$\sum_{y\in V'}(\gd_y-d_{\f}(y))^+\leq \go$.}
\enq

\mn

Of course for
\eqref{U|Q}
it is enough to show (as above using $\mN$ for $\bF\mmodels \mN$ and so on)
\beq{U|Q1}
\mbox{for any $\h\in \K^t\cap \mR$,
$~~\pr(\mN|\bH=\h) = 1-o(1)$}
\enq
and (our main point)
\beq{U|Q2}
\pr(\mV |\mA\mR\mQ) > 1-o(1).
\enq

We prove \eqref{U|Q1} and \eqref{Ubd} in Section~\ref{PLF2}
and \eqref{U|Q2} in Section~\ref{PLF3}, organizing in this way because
the proofs of \eqref{U|Q1} and \eqref{Ubd}
are slightly similar (mainly in their use of Lemma~\ref{ZZ'lemma}) and unrelated to the proof of
\eqref{U|Q2}.

\section{Proofs of \eqref{U|Q1} and \eqref{Ubd}}
\label{PLF2}

\begin{proof}[Proof of \eqref{U|Q1}]
We will show that \eqref{N3} follows (deterministically) from $\bH\in \mR$, while the other parts of $\mN$
are implied (again, deterministically) by the combination of $\h\in\mR$,
\beq{T1}
\mbox{$\bT$ is a matching}
\enq
and
\beq{T2}
\mbox{$\bT$ covers no $x$ for which
$d_{\bG}(x)\leq 1.5\eps D$.}
\enq
(Recall we are using $\bH$, $m$ and $D$ for $\bH_t$, $m_t$ and $D_m$.)
This gives \eqref{U|Q1} since (under $\bH\in \mR$)
the upper bound on $\tau$ in \eqref{tau} implies that each of
\eqref{T1}, \eqref{T2} holds with probability $1-o(1)$:
for \eqref{T1} this is a weak consequence of 
$\gD_\h=O(D_\h)$ (see \eqref{Rg1});
and
for \eqref{T2} it holds because
$\h\in \mR^3$ (with Observation~\ref{mR1Cor})
and the codegree condition
\eqref{Rg2} bound
the number of $x$'s in \eqref{T2} by $2rn^{2\gd}$.

\mn

Turning to the deterministic assertions preceding \eqref{T1}, 
we first note that \eqref{N2}
is immediate from $\bH\in \mR$ (specifically, \eqref{m3mT} and \eqref{Rg2})
and \eqref{T1}, while \eqref{N4} follows easily from $\h\in \mR^4$, \eqref{T1} and \eqref{T2}
(the first trivially implies \eqref{N4} with $\bG$ in place of $\f$, and
the others say the passage to $\bF$ doesn't affect this).

\mn

To get \eqref{N3} from $\bH\in \mR$, 
notice that 
\[
\exp[-(1-o(1))n^{2\gd}] ~> ~\pr(\bH_\ga ~\mbox{anemic})~\geq ~
\pr(\bF_\ga ~\mbox{anemic})\cdot (1-\ga)^{|\bH\sm\bF|}
\]
(the first inequality is $\bH\in \mR^3$ and the second is trivial).
Thus for \eqref{N3} it is enough to show 
\beq{PHFvs}
(1-\ga)^{|\bH\sm\bF|}= \exp[-o(n^{2\gd})]
\enq
---which isn't close:  we have (using \eqref{Rg1})
\beq{mm0}
|\bH\sm \bF|=|\bH\sm \bG|+\tau = O(D) +\tau =O(m/n)+\tau;
\enq
so the l.h.s. of \eqref{PHFvs} is (crudely)
\[
(1-\ga)^{O(m/n)+\tau}
=\left\{\begin{array}{ll}
\exp[-O(\frac{m_{_T}}{m}\{\frac{m}{n}+\tau\})]&\mbox{if $\ga < 1/2$ (say),}\\
~\\
\exp[-O(\frac{m}{n}+\tau)\log n]&\mbox{otherwise}
\end{array}\right.
\]
(the latter since $1-\ga>1/m$),
and \eqref{PHFvs} follows easily, 
using the upper bound on $\tau $ in \eqref{tau}
and $m_{_T}< n\log n$ (which when $\ga\geq 1/2$ also implies $m = O(n\log n)$).

\mn

Finally, we turn to the two cases of \eqref{N1}.
For $m$ violating \eqref{m.not.small}, we use Lemma~\ref{Psmallm}.
Here the assumption \eqref{xdHs} holds even for $\bF$---as follows from
$\bH\in \mR$ (specifically, \eqref{mR1cor} and 
\eqref{Rg2}), \eqref{T1} and \eqref{T2}---so
also for any $\g'\in \mH(\bF)$.  Thus the lemma gives
$\gb(\g'\cup \I)=1-o(1)$ whenever $\g'\cup \I\in \mL$
(with $\g'\in \mH(\bF)$ and $\I\in \mI(g)$), which in particular is true whenever 
$\g'\in \mH^*(\bF)$ and $\I\in \mI^*(g)$; so we have \eqref{N1} in this case.

\mn

For $m$ \emph{satisfying} \eqref{m.not.small},
recall from \eqref{mR1} 
[or, in Case 1, its specialization \eqref{betaH}]
that $\bH\in \mR^2$ implies the inequality in
\eqref{N1} for $\g'=\bG$ (since $\barvp(\bG)$ is the same as
$\barvp_Z(\bH)$);
so it is enough to show
\beq{phiG'phiG}
\vp(\g')\more \vp(\bG) \,\,\,\forall \g'\in \mH(\bF).
\enq
It's also easy to see that $\barvp(\bG)\more n^{-(2r+1)}\gb(\K)$ 
(as in \eqref{N1}) implies (say)
\beq{phiG}
\vp(\bG) \sim
\sum\{\gb(\bG\cup \I):\I\in \mI(g), \gb(\bG\cup\I)> n^{-(2r+2)}\gb(\K)\}.
\enq
But we claim that for any $\I$ as in \eqref{phiG}
and $\g'\in \mH(\bF)$,
\beq{bG'J}
\gb(\g'\cup \I)\more \gb(\bG\cup \I),
\enq
which (in view of \eqref{phiG}) gives \eqref{phiG'phiG}.

\mn

For \eqref{bG'J} we apply Lemma~\ref{ZZ'lemma} with ${\JJJ}=\bG\cup \I$ and ${\JJJ}'=\g'\cup \I$
(so $\eee \supseteq\bF\cup\I$, $\gs\leq\tau$ and $W\sub V'$).
Note that here the $\kappa$ of the lemma is zero, since
all $A_i$'s lie in $\bT$, so by 
\eqref{T2}
are not dangerous for ${\JJJ}$.
So the lemma's conclusion is \eqref{bG'J}, and we just need to check
its hypotheses (assuming ${\JJJ}\in \mL$, without which the r.h.s.\ of 
\eqref{bG'J} is zero):

First,
\eqref{mmT} holds because we assume \eqref{m.not.small},
and \eqref{Aimg} is given by \eqref{T1}.
Second,
${\JJJ}'\in \mL$ follows from ${\JJJ}\in \mL$,
using $\J\sm \J' \sub \bT$ with \eqref{T1} and \eqref{T2}.
Third,
\eqref{T1} implies
$d_{\eee}(x)\geq d_{\bG}(x)-1 $ for each $x\in V'$ ($\supseteq W$), which gives
\eqref{nodang} since
(using \eqref{m3mT} and \eqref{Rg2}, and noting $|\bH|=|\J|$)
\[
m>2rm_{_T} ~\Ra ~ [d_{\bH}(x)> 2\eps D_\J ~\forall x\in V]
~\Ra ~ [d_{\bG}(x)> (2\eps -o(1))D_\J ~\forall x\in V'].
\]
Last, \eqref{betaD} holds because
$\gb({\JJJ}) > \exp[-(1-o(1))n^\gd]$
(by Corollary~\ref{betacor}, since
$\I$ is as in \eqref{phiG}) and, by \eqref{N3}, 
$\pr(\bX_{\JJJ}~ \mbox{anemic})$ ($\leq \pr(\f_\ga~ \mbox{anemic})$) $ <\exp[-(1-o(1))n^{2\gd}]$.
\end{proof}

\begin{proof}[Proof of \eqref{Ubd}]
\emph{Implied constants in this argument do not depend on $\eps$ or $\theta$.}
We actually show
\beq{Ubd'}
\mbox{for any $\f\in \mN, ~~\pr(\mV|\bF=\f) = e^{-\gO(\theta\tau)}$}
\enq
(which is $n^{-\go(1)}$ by \eqref{tau}).
Here we use viewpoint (B).  The (natural) idea is roughly:  $\f$ determines the
weights $\ww_{\f}(\uu )$ (for all $U\in \K'$, though here we are only interested in $U\in \K'\sm \f$),
and $\mV$ then requires that $\bT$ be (pathologically) drawn almost entirely from $U$'s with weights
close to $\gz \Phi'$, though this group excludes a constant fraction of $\K'\sm \f$.

In fact \eqref{Ubd} would be more or less routine if we were choosing $\T$ \emph{uniformly}
rather than as in \eqref{reweightT}.  The crude 
comparison in Lemma~\ref{prmu} below will allow us to move between these two regimes.

\mn

Fix $\f\in \mN$ 
and  set $g=|\f| +\tau$, 
$\mI=\mI(g),\mI^*=\mI^*(g),\mH=\mH(\f)$ and $\mH^*=\mH^*(\f)$ 
(see \eqref{mI}-\eqref{mH*}).
We now regard the $\pr$ of \eqref{reweightT} as a probability measure on $\mH$---thus
\beq{Pg'}
\pr(\g')\propto \vp(\g') ~~\mbox{for $\g'\in \mH$}
\enq
---and use $\mu$ for uniform measure on $\mH$.
(It's perhaps worth noting---though this won't matter---that, unlike $\pr$, 
$\mu$ can assign positive probability to $\T$'s for which $\f\cup \T$ is not 
a possible value of $\bG$.)

\begin{lemma}\label{prmu}
With notation as above, if $\mX\sub\mH$ and
$\mu(\mX)=e^{-\gO(\theta\tau)}$, then
\[
\pr(\mX)=e^{-\gO(\theta\tau)}.
\]
\end{lemma}
\begin{proof}
Notice first that
\eqref{N4} and $\tau\geq \go$ (see \eqref{tau})
imply
\beq{munC}
\mu(\mH^*)> n^{-\go}.
\enq

For $m$ violating \eqref{m.not.small}
we now finish easily: combining \eqref{N1} and the trivial 
$\vp(\g')\leq |\mI^*|$ ($\forall \g'\in \mH$) with \eqref{munC} 
(and \eqref{Pg'}) gives $\pr(\mX) \less n^{\go}\mu(\mX)$ for any $\mX\sub\mH$;
and this gives the lemma since $\tau\gg \go\log n$ (again see \eqref{tau}).
So we assume from now on that $m$ satisfies \eqref{m.not.small}.

\mn

For $\g'\in \mH$, let
$\gl(\g')$ be the number of edges of $\g'\sm \f$ containing vertices $x$ with
\beq{xdF}
d_\f(x) < 1.5 \eps D.
\enq
By \eqref{N3} and
Observation~\ref{mR1Cor},
the number of such
vertices is less than $2n^{2\gd}$, implying that, for any $b$,
\beq{mu.kappa.h}
\mu(\g':\gl(\g')\geq b)=n^{-\gO(b)}
\enq
(Because (e.g.): if there \emph{are} $x$'s as in \eqref{xdF}, then 
\eqref{N2} bounds the fraction of members of $\K'\sm\f$ containing such $x$'s by
$O(n^{-1+2\gd})$, and the upper bound in \eqref{tau} then gives \eqref{mu.kappa.h}.)

\mn

Fix $\g^0\in \mH^*$ with $\vp(\g^0)$ minimum.  We will show that for
any $\g'\in \mH$,
\beq{vp.ratio}
\vp(\g')/\vp(\g^0) \less
n^{O(\eps\gl(\g'))}.
\enq
(Of course we can replace ``$\less$'' by ``$<$'' if $\gl(\g')\neq 0$.)

\mn

Before proving \eqref{vp.ratio} we show 
it gives Lemma~\ref{prmu} (for $m$ satisfying \eqref{m.not.small}).
Since (by \eqref{munC})
\[
\sum\{\vp(\g'):\g'\in \mH\} > n^{-\go}|\mH|\vp(\g^0),
\]
we have, for any $\g'\in \mH$ (using \eqref{vp.ratio}, and with $\g''$ running over $\mH$),
\[
\pr(\g')~=~\frac{\vp(\g')}{\sum\vp(\g'')} ~<~ \frac{n^{\go}}{|\mH|}\frac{\vp(\g')}{\vp(\g^0)}
~\less~\frac{n^{\go}n^{O(\eps\gl(\g'))}}{|\mH|}
=n^{\go}\mu(\g')n^{O(\eps\gl(\g'))}.
\]
Combining this with \eqref{mu.kappa.h}, $\tau\gg \go\log n$ 
and $\mu(\mX)=e^{-\gO(\theta\tau)}$ gives the desired bound:
with $\gl_0 =\theta\tau/\log n$,
\beq{Pg0}
\pr(\mX)\less
\mbox{$n^{\go}\left[\mu(\mX)n^{O(\eps\gl_0)} 
+ \sum_{b>\gl_0}\mu(\g':\gl(\g')=b)n^{O(\eps b)}\right]$}
=
e^{-\gO(\theta\tau)}.\qedhere
\enq

\begin{proof}[Proof of \eqref{vp.ratio}]
Set $\mI^0=\{\I\in \mI: \gb(\g'\cup \I) > n^{-(2r+2)}\gb(\K)\}$
($\sub \mI^*$)
and notice that it is enough to show
\beq{betaG0}
\gb(\g'\cup \I) \less n^{O(\eps\gl(\g'))}\gb(\g^0\cup \I) \,\,\,\forall ~\I\in \mI^0;
\enq
for if this is true then, recalling that $\f\in \mN$ implies
$\barvp(\g^0) \more n^{-(2r+1)}\gb(\K)$ (see \eqref{N1}), we have
\begin{eqnarray*}
\barvp(\g')&\leq& 
\mbox{$|\mI|^{-1}\sum\{\gb(\g'\cup \I):\I\in \mI^0\} + (1+o(1))n^{-1}\barvp(\g^0)$}\\
&\less& n^{O(\eps\gl(\g'))}\vp(\g^0)/|\mI| +n^{-1}\barvp(\g^0)
~\less~
n^{O(\eps\gl(\g'))}\barvp(\g^0).
\end{eqnarray*}
[In Case 1, \eqref{N1} says $\mI^0=\{\0\}$ ($=\mI$), and \eqref{betaG0} \emph{is}
\eqref{vp.ratio}.]

For \eqref{betaG0} we will again use Lemma~\ref{ZZ'lemma}, now with
${\JJJ}=\g'\cup \I$ and ${\JJJ}'=\g^0\cup \I$
(and $\eee=(\g'\cap\g^0)\cup \I\supseteq \f$),
so should check hypotheses:
first, 
\eqref{mmT} holds since we assume \eqref{m.not.small},
and \eqref{Aimg} is given by \eqref{T1} (since $\J\sm \eee\sub \T$);
second, assuming (as we may) that ${\JJJ}\in \mL$, we have
${\JJJ}'\in \mL$, since $\g^0\in \mH^*$ 
and
${\JJJ},{\JJJ}'$ agree on edges meeting $Z$;
third, \eqref{nodang} holds since (for $m>2rm_{_T}$) \eqref{N2} gives
$
d_\eee(x)\geq d_\f(x) > 1.5\eps D$
($=1.5\eps D_{\JJJ})
$
for $x\in V' ~ (\supseteq W)$;
last, \eqref{betaD} follows from $\I\in \mI^0$ (with Corollary~\ref{betacor})
and \eqref{N3} (with
$\pr(\bX_{\J}~ \mbox{anemic}) \leq \pr(\bF_\ga~ \mbox{anemic})$).

So the lemma applies and we just need to check that $\kappa$ 
(the number of edges of $\J\sm\J'$ containing vertices dangerous for $\J$)
is at most $\gl(\g')$;
but this is true because $\JJJ\sm \JJJ'\sub \g'\sm \f$ and any $x$ that is dangerous for $\JJJ$
satisfies \eqref{xdF} (since 
$d_\f(x)\leq d_{\JJJ}(x) $).
\end{proof}
This completes the proof of Lemma~\ref{prmu}.\end{proof}

We return to \eqref{Ubd'}, which by Lemma~\ref{prmu}
will follow from its ``$\mu$-version,'' \emph{viz.}
\beq{Ubd''}
\mbox{\emph{for any $\f\in \mN, ~~\mu(\mV|\bF=\f) = e^{-\gO(\theta\tau)}.$}}
\enq
A small complication here is that $\f$ doesn't determine the ``target" $\gz \Phi'$
appearing in $\mV$.
Among several ways of dealing with this, the following seems nicest.

Given $\f$, 
let $\uu _1,\ldots$ be an ordering of $\K'\sm \f$ with
$\ww_{\f}(\uu _1)\leq \ww_{\f}(\uu _2)\leq \cdots$,
and let $\YYY$ and $\ZZZ$ be (resp.) the first and last $\theta|\K'\sm \f|/3$ of the $\uu _i$'s.
Then, \emph{whatever} $\Phi'$ turns out to be, the second part of $\mV$ requires that
at least one of $\YYY$, $\ZZZ$ be contained in
\[
\W:=\{\uu \in \K'\sm \f: \ww_{\f}(\uu ) \neq (1\pm \theta)\gz \Phi'\}
\]
(or
$|\W|< |\YYY|+|\ZZZ|<2(\theta |\K'\sm \bG|+\tau)/3<\theta |\K'\sm \bG|$,
the last inequality since \eqref{sillyp} and \eqref{tau} imply $\tau\ll |\K'\sm \bG|$).
But then, since (now regarding $\mu$ as the law of $\bT$)
\[
\E_\mu|\bT\cap \YYY|=\E_\mu|\bT\cap \ZZZ| =\theta\tau/3
\]
and $\theta$ is fixed, Theorem~\ref{T2.1} bounds
the probability that the first part of $\mV$ 
holds by (say)
\[
\mu(\max\{|\bT\cap \YYY|,|\bT\cap \ZZZ|\}< \theta\tau/4) =e^{-\gO(\theta\tau)}.
\]
\end{proof}

\section{Proof of \eqref{U|Q2}}
\label{PLF3}

(We continue to use $\bH=\bH_t$, $m=m_t$ and $D=D_m$.)
We now need to pay some attention to parameters.
We first observe that if $\h\mmodels\mA\mR^0$ and $\g=\h-Z$
(with $Z$ either empty or as in \eqref{WZ}), then there is
$\gc=o(1)$ (depending on the $o(n)$ in $\mA$ and, in $\mR^0$,
the (explicit or implicit)
$o(\cdot)$'s in \eqref{Rg0} and \eqref{Rg2},
and the implied constants in \eqref{Rg1}),
such that for each $U\in \K'$ with (say)
\beq{Phi*bd}
(\ww_\g(U)=) \,\,\,\Phi(\g-U)> \Phi(\g)n^{-r},
\enq
$\g^*:=\g-U$ and $\Phi^*:=\Phi(\g^*)$
satisfy
\beq{heavysum}
\sum\{\ww_{\g^*}(A):A\in \g^*, \ww_{\g^*}(A)\neq (1\pm \gc)\Phi^*/D\}< \gc n\Phi^*.
\enq
To see this, recall from the remarks following 
Lemma~\ref{RCEFB}
that each relevant $\g^*$ satisfies $\mA\mR^0$,
so also $\mE$ by Lemma~\ref{LemmaE}.
But then $\g^*$ contains $(1-o(1))|\g^*|\sim nD/r$ edges $A$ with 
$\ww_{\g^*}(A)\sim \Phi^*/D_{\g^*}\sim \Phi^*/D$,
with (both) asymptotics following easily from $\g\mmodels\mR^0$
(see \eqref{Rg2});
so such edges account for all but a $o(1)$-fraction of the total weight $\Phi^*(n-r)/r\sim \Phi^*n/r$.
This gives \eqref{heavysum} for a suitable $\gc=o(1)$.

\mn

We now choose $\tau =\nu \min\{\log n,D\}$ ($\sim\nu \log n$)---noting that then
\beq{gznu}
\gz ~~(=e^{-\tau/D}) ~ \geq e^{-\nu}
\enq
---together with $M$ and $\eta$,
satisfying
\beq{param1}
\log n\gg\nu\gg \go
\enq
(which is \eqref{tau});
\beq{param2}
e^{-\nu}\gg \gc;
\enq
\beq{param3}
\tau\gg M\left\{\begin{array}{ll}
\gg \gc \tau,\\
> 1+\gc;
\end{array}\right.
\enq
and
\beq{param4}
e^{-\nu}\gg \eta\gg \sqrt{\tau M}/\log n.
\enq
Note this is possible:
we may choose $\nu\ra\infty$ as slowly as we like (which in particular gives
\eqref{param1} and \eqref{param2}); we then want to choose $M$ as in \eqref{param3} satisfying (to
leave room for $\eta$)
$
e^{-\nu}\gg \sqrt{\tau M}/\log n;
$
and this is possible if $e^{-\nu}\gg \max\{\nu\sqrt{\gc},\sqrt{\nu/\log n}\}$,
which is true for a slow enough $\nu$.

\mn

For the proof of \eqref{U|Q2} we use viewpoint (A) (choose $\bH$---so also $\bG$---and then $\bT$).
We assume we have chosen $\bH=\h$, with $\h\mmodels\mA\mR$ and $\g:=\h-Z\mmodels\mQ$; 
so $\pr$ now
refers just to the choice of $\bT$, and \eqref{U|Q2} will follow from
\beq{G,T}
\pr((\g,\bT)\mmodels \mV) = 1-o(1).
\enq
It will be enough to show that for $\uu \in \K'$ as in \eqref{Phi*bd}
(i.e.\
$\ww_\g(\uu )> \Phi(\g)n^{-r}$),
\begin{eqnarray}
\pr(\ww_{\bF}(\uu ) \sim \gz\ww_\g(\uu )) = 1-o(1)&\mbox{if $U\in \K'\sm \g$,}
\label{WGTZ1}\\
\pr(\ww_{\bF}(U ) \sim \gz\ww_\g(U )|U\in \bT) = 1-o(1)&\mbox{if $U\in \g$.}
\label{WGTZ2}
\end{eqnarray}

\mn

Before proving this we show that it does give \eqref{G,T}.
If $\g$ satisfies $\mQ$ then for a suitable $\vs=o(1)$,
\beq{GO}
|\{A\in \g:\ww_\g(A)\neq (1\pm \vs)\Phi'\}| \ll |\g|
\enq
(where $\Phi'=\Phi(\g)/D_\g$; see \eqref{Phi'}).
Thus, with $\g^0$ the set in \eqref{GO}, we have
$
\E |\T\cap\g^0|=\tau |\g^0|/|\g|\ll \tau,
$
so
\[
\mbox{$|\bT\cap \g^0|\ll\tau~$  w.h.p.}
\]
(by Theorem~\ref{T2.1} or just Markov's Inequality).
But for the first part of $\mV$ to fail we must have either
$|\T\cap\g^0|=\gO(\tau)$,
which we have just said occurs with probability $o(1)$, or
\[   
|\{A\in \T\sm\g^0: \ww_{\f}(A)\not\sim\gz\ww_\g(A)\}|=\gO(\tau),
\]   
which has probability $o(1)$ by \eqref{WGTZ2} (and Markov).

Similarly, failure of the second part of $\mV$ implies
\beq{WZLarge}
\ww_{\f}(U) = (1\pm \theta)\gz\ww_\g(U)
\not\sim\gz\ww_\g(U)
\enq
for at least $\theta|\K'\sm \g|$ of those
$\uu $'s in
the second part of $\mQ$ that satisfy
\beq{WZlarge}
\ww_\g(\uu )> (1-\theta)\gz \Phi' > n^{-o(1)}\Phi(\g)/D
\enq
(since those failing \eqref{WZlarge} \emph{cannot} satisfy \eqref{WZLarge};
for the second bound in \eqref{WZlarge} see \eqref{gznu} and \eqref{param1}).
But since the bound in \eqref{WZlarge} is larger than the one in \eqref{Phi*bd},
\eqref{WGTZ1} implies that the probability that \eqref{WZLarge}
holds for such a set of $U$'s is $o(1)$.

\mn

Finally, we prove \eqref{WGTZ1}; the proof of \eqref{WGTZ2} is almost
literally the same and is omitted.
(Note the probability in \eqref{WGTZ2} is just
$\pr(\ww_{\g\sm \bT_0}(U ) \sim \gz\ww_\g(U ))$, with $\bT_0$ uniform from
$\C{\g\sm\{U\}}{\tau-1}$.)

\begin{proof}[Proof of \eqref{WGTZ1}.]

We now fix $U $ as in \eqref{Phi*bd} (and recall $\g^*=\g-\uu $ and
$\Phi^*=\Phi(\g^*)$).
(We will, pedantically, keep track of the microscopic numerical differences
between Cases 1 and 2---in Case 2 the number of vertices is $n-r$
and $|\g|$ is not exactly $m$---but stress they are wholly irrelevant.)

\mn

Say $A\in\g$ is \emph{heavy} if $A\in\g^*$ and $\ww_{\g^*}(A)> M\Phi^*/D$, and note that by
\eqref{heavysum} (and $M> 1+\gc$; see \eqref{param3}),
\beq{fewheavies}
\mbox{the number of heavy edges in $\g$ is less than $\gc nD/M=\gc mr/M$,}
\enq
implying
\beq{Theavy}
\pr(\mbox{$\bT$ contains a heavy edge}) \less \gc \tau r/M =o(1)
\enq
(see \eqref{param3}; we need ``$\less$'' because we only have $|\g|\sim m$).  
So it is enough to show \eqref{WGTZ1} conditioned on
\beq{noheavy}
\{\mbox{$\bT$ contains no heavy edges}\}.
\enq
We will instead show a slight variant, replacing
$\bT$ by $\bT'=\{A_1\dots A_\tau\}$, with the $A_i$'s chosen uniformly and \emph{independently}
from the non-heavy edges of $\g$; thus:
\beq{WGTZ'}
\pr(\ww_{\g\sm\bT'}(\uu ) \sim \gz \ww_\g(\uu )) =1-o(1).
\enq
Of course this suffices:  we may couple $\bT$ (conditioned on \eqref{noheavy})
and $\bT'$ so they agree whenever the
edges of $\bT'$ are distinct, which occurs w.h.p.\ (more precisely, with probability greater than
$1-\tau^2/m$), and the probability in \eqref{WGTZ1} is then at least
the probability in \eqref{WGTZ'} minus $\pr(\bT'\neq\bT)$.

\mn

For the proof of \eqref{WGTZ'}, let
\[
X= X(A_1\dots A_\tau)  =\Phi(\g^*\sm\{A_1\dots A_\tau\})= \ww_{\g\sm \bT'}(U).
\]
Since $\eta\ll \gz $ (see \eqref{gznu} and
\eqref{param4}), \eqref{WGTZ'} will follow from
(recall $\ww_\g(U)=\Phi^*$)
\beq{EPhiX}
\E X\sim \gz \Phi^*
\enq
and
\beq{prXE}
\pr(|X-\E X|>\eta \Phi^*)=o(1).
\enq
\begin{proof}[Proof of \eqref{EPhiX}]
Let $M_i$ run through the p.m.s of $\g^*$ and let $x_i$ be the number of
heavy edges in $M_i$.
Then with $m'$ the number of non-heavy edges in $\g$, we have
(with, irrelevantly, $\iota=j$ in Case $\textrm{j}$ for $\textrm{j}=1,2$),
\[
\mbox{$\E X = \sum_i (1-(n/r-\iota-x_i)/m')^\tau$}
\]
and, by \eqref{heavysum},
\beq{sumxi'}
\sum x_i =\sum\{\ww_{\g^*}(A):\mbox{$A\in \g^*$, $A$ heavy}\}<\gc n\Phi^*.
\enq
These imply, with $\varrho =(n/r-\iota)/m'$,
\begin{eqnarray}
(1-\varrho)^\tau \Phi^*&\leq & \mbox{$\E X
~<~ \sum_i e^{-(\varrho-x_i/m') \tau}$}\label{EXsum}\\
&< &\left[e^{-\varrho\tau}+ \gc rn/(n-\iota r)\right]\Phi^*.\nonumber
\end{eqnarray}
Here the last inequality follows from \eqref{sumxi'} and convexity of the
exponential function, which imply that the sum in \eqref{EXsum} is at most what it would be
with $\frac{\gc n\Phi^*}{n/r-\iota} =\frac{\gc rn\Phi^* }{n-\iota r}$ of the $x_i$'s equal to 
$n/r-\iota$
and the rest (the number of which we just bound by $\Phi^*$) equal to zero.

In view of \eqref{EXsum}, \eqref{EPhiX} will follow from
\beq{1-vs}
(1-\varrho)^\tau\sim e^{-\varrho\tau}\sim e^{-\tau/D} ~~(=\gz)
\enq
(and $\gc \ll \gz$, which is given by \eqref{gznu} and \eqref{param2}).
For the two parts of \eqref{1-vs} we need (resp.) $\varrho^2\ll 1/\tau$ and
$|\varrho-1/D|\ll 1/\tau$.  The first of these follows from \eqref{fewheavies}
(which gives $m'\sim m$, though here $m'=\gO(m)$ would suffice)
and \eqref{tau}.
For the second, now using \eqref{fewheavies} more precisely
 (and recalling $D=mr/n$), we have
\[
\left| \frac{n/r-\iota}{m'}-\frac{n/r}{m}\right|
\leq \frac{\iota}{m'}+\frac{n}{r}~\frac{m-m'}{mm'} < \frac{\iota}{m'}+\frac{n}{r}\frac{\gc r}{Mm'}\ll \frac{1}{\tau},
\]
with the last inequality a (weak) consequence of \eqref{param3}.
\end{proof}

\begin{proof}[Proof of \eqref{prXE}]

We
consider the (Doob) martingale
\beq{Doob}
X_i=X_i(A_1\dots A_i)=\E[X|A_1\dots A_i]
\,\,\,(i=0\dots \tau),
\enq
with difference sequence $Z_i=X_i-X_{i-1}$
($i\in [\tau]$)
and $Z =\sum Z_i$ ($=X-\E X$).
For the next little bit we use $\E_S$ for expectation
with respect to
$(A_i:i\in S)$.

Given $A_1\dots A_{i-1}$ we may express
\beq{ZW}
Z_i=\E W-W,
\enq
where $\E$ refers to $A$ chosen uniformly from the non-heavy edges of $\g$ and
\beq{W(A)}
W(A) ~=~\E_{[i+1,\tau]}\Phi(\g^*\sm\{A_1\dots A_{i-1},A_{i+1}\dots A_\tau\})
\,\,\,\,\,\,\,\,\,\,\,\,\,\,\,\,\,\,\,\,\,\,\,\,\,\,\,\,\,\,\,\,
\,\,\,\,\,\,\,\,
\enq
\[
\,\,\,\,\,\,\,\,\,\,\,\,\,\,\,\,\,\,\,\,\,\,\,\,\,\,\,\,\,\,\,\,\,\,\,\,\,\,\,\,
-\E_{[i+1,\tau]}\Phi(\g^*\sm\{A_1\dots A_{i-1},A,A_{i+1}\dots A_\tau\}).
\]

\mn
For \eqref{ZW} just notice that
\[
X_i(A_1\dots A_{i-1},A) =\E_{[i+1,\tau]}\Phi(\g^*\sm\{A_1\dots A_{i-1},A,A_{i+1}\dots A_\tau\}),
\]
while
\[
X_{i-1}(A_1\dots A_{i-1}) =\E_{[i,\tau]}\Phi(\g^*\sm\{A_1\dots A_{i-1},A_i,A_{i+1}\dots A_\tau\}).
\]
(The first term on the r.h.s.\ of \eqref{W(A)} is chosen to give \eqref{WAWA} and,
not depending on $A$, doesn't affect \eqref{ZW}.)

We also have
\beq{WAWA}
0\leq W(A) \leq ~\ww_{\g^*}(A),
\enq
since these bounds hold even if we remove the $\E$'s in \eqref{W(A)}.
Thus $W$ satisfies the conditions in Proposition~\ref{EeZ} with $b=M\Phi^*/D$ and 
$a\sim\Phi^*/D$ (the latter since
$|\g|^{-1}\sum_{A\in \g}\ww_{\g^*}(A) =|\g|^{-1}\Phi^*(n/r-\iota) 
\less\Phi^*/D$---note $\ww_{\g^*}(A):=0$
if $A\in \g\sm\g^*$---and averaging instead only over non-heavy edges can only decrease this).
So for any
\beq{gzbds}
\vt \in [0,(2b)^{-1}],
\enq
we may apply Lemma~\ref{Azish} to each of $Z$, $-Z$, using Proposition~\ref{EeZ}
(with \eqref{ZW}) to bound the factors in
\eqref{EeeZ}, yielding
\[
\max\{\E e^{\vt Z},\E e^{-\vt Z}\}\leq e^{\tau \vt^2 ab} 
= \exp[(1+o(1))\tau\vt^2 M(\Phi^*/D)^2]
\]
and, for any $\gl>0$,
\beq{maxprZ}
\max\{\pr(Z>\gl),\pr(Z<-\gl)\} < \exp[(1+o(1))\tau\vt^2 M(\Phi^*/D)^2-\vt \gl].
\enq
For \eqref{prXE} we use \eqref{maxprZ} with $\gl =\eta \Phi^*$ and
\beq{gzvals}
\vt =\min\left\{\frac{\eta\Phi^*}{2\tau M(\Phi^*/D)^2}, \frac{D}{2M\Phi^*}\right\}
=\frac{D}{2M\Phi^*}\min \left\{ \frac{\eta D}{\tau},1\right\}
\enq
(the first value in ``min" essentially
minimizes the r.h.s.\ of \eqref{maxprZ}
and the second enforces \eqref{gzbds}), and should show that the exponent in \eqref{maxprZ}
is then $-\go(1)$.

Suppose first that $\eta D\leq \tau$, so $\vt$ takes the first value(s) in \eqref{gzvals}.
Then the negative of the exponent in \eqref{maxprZ} is asymptotically
(using \eqref{param4} and $D\more \log n$)
\[
\frac{(\eta\Phi^*)^2}{4\tau M(\Phi^*/D)^2} = \frac{\eta^2D^2}{4\tau M}
=\go( 1).
\]
If instead $\eta D>\tau$,
then $\vt = D/(2M\Phi^*)$ and the exponent in \eqref{maxprZ} is
\[
(1+o(1))\frac{D^2}{(2M\Phi^*)^2}\tau M\left(\frac{\Phi^*}{D}\right)^2 
-\frac{D}{2M\Phi^*}\eta \Phi^*
=(1+o(1)) \frac{\tau}{4M}-\frac{\eta D}{2M}
= -\go(1),
\]
where we used $\eta D>\tau$ and, from \eqref{param3}, $\tau\gg M$.
\end{proof}
This completes the proof of \eqref{WGTZ1}.\end{proof}

\section{Foundation}\label{Foundation}

Finally, we return to the assertions 
listed at the end of Section~\ref{SecR} 
that will complete the proof of \eqref{Ri}. 
As suggested earlier,
this is currently a much longer story
than it seems ought to be necessary, but we do the best we can,
as usual aiming for ``simplicity'' rather than strongest statements.

Most of of this involves behavior at $m_{_T}$.  
This ``foundation" is covered in Sections~\ref{Configs}-\ref{Degrees}, 
with the final points needed for \eqref{Ri} mostly in Section~\ref{AppR}.

\subsection{Configurations and simplicity}\label{Configs}

\mn
With $V=[n]$, the \emph{degree sequence} of $\h\sub\K$ is $\ud(\h)=(d_\h(1)\dots d_\h(n))$.
In what follows $\ud$ is always in
\beq{sss}
\mbox{$\sss
= \{(d_1\dots d_n): \sum d_i =m r\} $}
\enq
($m$ for now unspecified),
and we set 
\beq{Kud}
\K(\ud)= \{\h\sub \K: \ud(\h)=\ud\} .
\enq

We will work with the hypergraph version of the ``configuration model"
of Bollob\'as \cite{Boll79} (see \cite{Wormald} for a good discussion
of the model and antecedents).
Let $T$ be a set of size $mr$ 
and $T_1\cup\cdots \cup T_n$ a partition of
$T$ (into \emph{pre-verts}) with $|T_j|=d_j ~\forall j$
(so $mr=\sum d_i$).
A \emph{configuration} is an (\emph{unordered}) partition of $T$ into \emph{pre-edges} of size $r$; it
is \emph{simple} if
\beq{nopreedge}
\mbox{no pre-edge meets any pre-vert more than once}
\enq
and
\[
\mbox{no two pre-edges meet exactly the same pre-verts.}
\]

The projection $\pi:T\ra V$ given by $\pi(T_j) =\{j\}$ $\forall j$
maps each \emph{simple} configuration to some $\h\in\K(\ud)$, and for any such $\h$ we have
\[
\mbox{$|\pi^{-1}(\h)| =\prod d_j!$.}
\]
Thus
\beq{Kd}
\mbox{$|\K(\ud)| =\Psi\gc(\ud) (\prod d_j!)^{-1},$}
\enq
where $\Psi=\Psi(m,r)$
is the number of configurations (which of course depends only on $m$ and $r$)
and $\gc(\ud)$ is the probability that a uniformly
chosen configuration is simple.
The (easily calculated) $\Psi$
is irrelevant here, since 
we are only interested in ratios, but
we \emph{will} need some crude information on the $\gc(\ud)$'s.
(Much better estimates can be gotten by
adapting the switching methods of
McKay and Wormald; see \cite{McKay,MW} or, again, \cite{Wormald}.)

\mn

Since it costs nothing to do so, and perhaps helps clarify what's relevant, we state
our basic result here in some generality,
assuming the setup in the paragraph containing \eqref{nopreedge},
with $r\geq 3$ fixed and $D=\max d_i$.

\begin{lemma}\label{Lgamma}
If 
\beq{Dmn}
m^{2r-3}> n^{r-1}D^{2r-1},
\enq
then $\gc(\ud)=e^{-O(D)}$.
\end{lemma}
\nin
(We will use this with $D=n^{o(1)}$---or, really, with $D$
growing at most a little faster than $\log n$---so \eqref{Dmn} won't be an issue.)

\begin{proof}[Proof of Lemma~\ref{Lgamma}]
Here we think of configurations
in terms of maps, as follows.
Let $E=E_1\cup\cdots \cup E_m$, with the $E_i$'s disjoint $r$-sets.
A bijection $\gs:E\ra T$ gives the configuration $\{\gs(E_1)\dots \gs(E_{m_{_T}})\}$,
and we say $\gs$ is simple if the configuration is.
Thus for a uniform $\bgs$, $\gc(\ud)=\pr(\mbox{$\bgs$ is simple})$ and Lemma~\ref{Lgamma}
becomes
\beq{13.1'}
\mbox{\emph{under the assumptions of Lemma~\ref{Lgamma},
$~\pr(\bgs ~\text{is simple})=e^{-O(D)}$.}}
\enq

The proof of this
uses the Lov\'asz Local Lemma
\cite{Erdos-Lovasz} in the following form (see \cite{AS}, Lemma 5.1.1
and the remark beginning near the bottom of p.\ 71).

\begin{lemma}\label{LLL}
Let $A_1\dots A_s$ be events in a probability space,
$\gG$ a graph on $[s]$ (thought of as a set of edges),
and $x_1\dots x_s\in [0,1)$.
Suppose that for any $i\in [s]$ and $S\sub [s]\sm (\{i\}\cup \{j:ij\in \gG\})$,
\beq{LLLhyp}
\mbox{$
\pr(A_i|\wedge_{j\in S}\bar{A}_j)\leq x_i\prod_{ij\in \gG}(1-x_j).
$}
\enq
Then
\[   
\mbox{$
\pr(\wedge_i\bar{A}_i)\geq \prod(1-x_i).
$}
\]    
\end{lemma}

\nin
Our use of this, which is reminiscent of
\cite{ES} (or see \cite[Sec.~5.6]{AS}), depends on
the following observation.

We consider bijections
$\gs:[N]\ra [N]$,
each for now regarded as a set of $N$ cells $(i,\gs(i))$ of an $N\times N$ array $M$.
We use \emph{pattern} to mean a set of cells (in $M$), no two on a line
(i.e.\ row or column), and define patterns $X,Y$ to be adjacent ($X\sim Y$)
if some line meets both.  Let $\bgs$ be a uniform bijection
and, for a pattern $X$, let
$A_X$ be the event $\{\bgs\supseteq X\}$,
noting that (with $(a)_t=a(a-1)\cdots (a-t+1)$)
\beq{pAX}
\pr(A_X) = 1/(N)_{|X|}.
\enq
\begin{prop}\label{ESprop}
If $X,X_1\dots X_t$ are patterns with $X\not\sim X_i ~\forall i$, then
\[
\mbox{$\pr(A_X|\wedge_{i=1}^t\bar{A}_{X_i})\leq \pr(A_X).$}
\]
\end{prop}

\begin{proof}
We may assume $X=\{(i,i):i\in [k]\}$, so 
\beq{k1N}
\cup X_j\sub \{k+1\dots N\}^2.
\enq
Set $\R=\{\gs:\gs(i)=i~\forall i\in [k]\}$ (so $A_X=\{\bgs\in \R\}$).
With $B=\wedge_{i=1}^t\bar{A}_{X_i}$, it is enough to exhibit, for any distinct $j_1\dots j_i\in [N]$,
an injection $\psi:\R \ra \T:=\{\gs:\gs(i)=j_i~\forall i\in [k]\}$ satisfying
\beq{sigmaB}
\gs\in B\Ra \psi(\gs)\in B.
\enq

Here it's convenient to interpret a bijection $\gs$ as a perfect matching of $K_{N,N}$
(whose vertex set we regard as two copies of $[N]$).
For $\gs\in\R$ and $\tau:=\{(i,j_i):i\in [k]\}$ 
(also thought of as a matching of $K_{N,N}$), the components of
$\gs\cup \tau$ are paths and cycles, each alternating with respect to $(\gs,\tau)$
(with the obvious meaning; in particular an edge of $\gs\cap\tau$ is considered
an alternating 2-cycle),
and with the ends of the paths the vertices not covered by $\tau$.
(Some---many---of these paths may be single edges of $\gs$.)

We then take $\psi(\gs)$ to consist of $\tau$
together with all edges that complete path components of $\gs\cup \tau$ to cycles.
It is straightforward to check that $\psi$ has the desired properties;
that is, it maps $\R$ injectively to $\T$ and satisfies \eqref{sigmaB}.
(Both of these follow from the observation that the edges of $\psi(\gs)$ not in $\gs$ 
are precisely those not of the form $(i,i)$ that meet at least
one of the two copies of $[k]$.)\end{proof}

We return to \eqref{13.1'}.
We will use Proposition~\ref{ESprop} with $N=mr$ and $E$ and $T$ our two copies of $[N]$
(so $M$ is an $E\times T$ array).
Define a \emph{block} to be a subarray indexed by some $E_i\times T_j$
(denoted $B_{ij}$).
We consider two types of patterns (``loops" and ``repeats"):

\begin{itemize}

\item
[(L)]
two cells in the same block;

\item
[(R)]
for some $i\neq j$ and distinct $l_1\dots l_r$,
$2r$ cells, one in each of the blocks indexed by $\{i,j\}\times \{l_1\dots l_r\}$.

\end{itemize}

\mn
Then a bijection $\gs:E\ra T$ is simple iff it contains none of these patterns, and
\beq{degobs}
\mbox{each cell lies in $O(D)$ patterns of type L and $O(mn^{r-1}D^{2r-1})$ of type R.}
\enq

Now let $X_1\dots X_s$ run over patterns of types L and R, write $A_i$ for $A_{X_i}$,
and let $\gG$ be the graph on $[s]$ with adjacency corresponding to adjacency of 
patterns as in Proposition~\ref{ESprop}
(so $i\sim j$ iff $X_i\sim X_j$).
Since lines have size $mr =O(m)$ (and patterns have size $O(1)$),
each pattern is adjacent to $O(mD)$ patterns of type L and $O(m^2n^{r-1}D^{2r-1})$
of type R.
So if we take (say)
\[
x_i=\left\{\begin{array}{ll}
x:=2N^{-2}&\mbox{if $X_i$ is of type L,}\\
y:=2N^{-2r}&\mbox{if $X_i$ is of type R,}
\end{array}\right.
\]
then each of the products $\prod_{ij\in \gG}(1-x_j)$ in
\eqref{LLLhyp} is 
\[
(1-x)^{O(mD)}(1-y)^{O(m^2n^{r-1}D^{2r-1})} \sim 1
\]
(the asymptotic following from \eqref{Dmn}),
which with \eqref{pAX} implies \eqref{LLLhyp}.

Thus, since \eqref{degobs}
bounds the numbers of type L and R patterns by
$O(m^2D)$ and $O(m^3n^{r-1}D^{2r-1})$ respectively,
Lemma~\ref{LLL} gives (again using \eqref{Dmn})
\beq{gs.simple.last}
\pr(\mbox{$\bgs$ is simple}) \geq
(1-x)^{O(m^2D)}(1-y)^{O(m^3n^{r-1}D^{2r-1})}=  e^{-O(D)},
\enq
which is \eqref{13.1'}
and completes the proof of Lemma~\ref{Lgamma}.
(Note \eqref{gs.simple.last} fails for $r=2$---as it should, since 
Lemma~\ref{Lgamma} is not true in this case.)
\end{proof}

\subsection{Degrees}\label{Degrees}

In this section \emph{only} we 
take $m=m_{_T}$ and $\bH=\bH_T$.  
Notice that we may choose
$\bH$ by first choosing $\ud:=\ud(\bH)\in \sss$ 
($= \{(d_1\dots d_n): \sum d_i =m r\} $ as in \eqref{sss})
and then
$\bH$ itself uniformly from $\K(\ud)$ (see \eqref{Kud}). 

Now thinking of the law of $\ud(\bH)$, we set 
\[
\Ll=\{\ud\in \sss:d_i\geq \gd_i ~\forall i\}.
\]
(The $\gd_i$'s are our usual $\gd_x$'s, so are asymptotic to $\eps\log n$.)
Then with $\ph$ the probability measure on $\sss$ given by
\[   
\prh(\ud)\propto |\K(\ud)|,
\]   
we have
\[   
\pr_h(\Ll) =\gb(\K)
\]   
and 
\beq{prudH}
\mbox{$\pr(\ud(\bH)=\ud)= \prh(\ud|\Ll)$}
\enq
($= \ph(\ud)/\ph(\Ll)$ if $\ud\in \Ll$).
We compare $\pr_h$ to the probability measure $\pr_u$ on $\sss$ given by
\[
\mbox{$\pr_u(\ud)\propto (\prod d_i!)^{-1}$.}
\]
Thus $\pr_u(\ud)$ is the probability that $mr$ ($= m_{_T}r$) balls
placed uniformly and independently in urns $U_1\dots U_n$ produce the
occupation statistics $\ud$, and (by \eqref{Kd})
\[
\pr_h(\ud)\propto \gc(\ud)\pr_u(\ud).
\]

\mn

For better understanding the law of $\ud(\bH)$ (as in \eqref{prudH}) we will
use 
Lemma~\ref{Lgamma}
and the following easy observations, whose verifications we omit.
\begin{obs}\label{Obs0}  Under $\pr_u$, for any $L\sub [n]$, with $|L|=l$,
\[
\mbox{$\sum_{i\in L}d_i\sim {\rm Bin}(mr, l/n).$}
\]
\end{obs}
\begin{obs}\label{Obs1}
If $\gc(\ud)> \xi $ for all $\ud\in {\JJJ}\sub \sss$, then for any $\eee\sub \sss$,
\[
\ph(\eee)/\ph({\JJJ})< \xi^{-1}\pu(\eee)/\pu({\JJJ}).
\]
\end{obs}
\begin{obs}\label{Obs2}
For any $i\in [n]$, $\pu(\Ll|d_i=k)$ is decreasing on $\{k\geq \gd_i\}$,
implying that, for any $J\geq \gd_i$,
\begin{eqnarray*}    
\pu(d_i\geq J|\Ll) &=&
\frac{\sum_{k\geq J}\pu(d_i=k)\pu(\Ll|d_i=k)}{\sum_{k\geq \gd_i}\pu(d_i=k)\pu(\Ll|d_i=k)}
\\
&\leq&
\frac{\pu(d_i\geq J)}{\pu(d_i\geq \gd_i)}
~\sim ~\pu(d_i\geq J).
\end{eqnarray*}
\end{obs}
\nin
(The initial assertion is a trivial coupling argument
and, since $m\sim (n/r)\log n$, the ``$\sim$" is a tiny consequence of Observation~\ref{Obs0} and Theorem~\ref{T2.1}.)

\mn

We next note that Theorem~\ref{T2.1} and Observation~\ref{Obs0} give  (for any $i$,
using $\E d_i =mr/n\sim \log n$)
\[
\pu(d_i>3\log n)< \exp\left[-(1-o(1))\tfrac{4\log n}{2(1+2/3)}\right]=n^{-6/5+o(1)}.
\]
Thus, now using Observation~\ref{Obs2},
\beq{mdiL}
\pu(\max d_i> 3\log n|\Ll)<n^{-1/5+o(1)} = o(1).
\enq

\begin{lemma}\label{LClogn}
For large enough $\kappa$,
\beq{LCl}
\pr(\max d_i(\bH)> \kappa\log n)< n^{-\kappa}.
\enq
\end{lemma}

\begin{proof}
Set
\[
\mbox{${\JJJ}=\{\max d_i\leq 3\log n\}\wedge\Ll~$ and 
$~\eee=\{\max d_i>\kappa\log n\}\wedge\Ll$.}
\]
Lemma~\ref{Lgamma} gives
\beq{maxdi3}
\mbox{if 
$\max d_i\leq 3\log n $ then $\gc(\ud)> n^{-K}$}
\enq

\nin
for some fixed $K$,
while Observation~\ref{Obs0} and Theorem~\ref{Cher'} imply
\beq{nK2}
\pu(d_i >\kappa\log n)< n^{-2\kappa }
\enq
for large enough $\kappa$ (the actual bound being essentially $n^{-\kappa\log (\kappa/e)}$).
The l.h.s.\ of \eqref{LCl} is then, again for large enough $\kappa$,
\begin{eqnarray*}
\ph(\eee)/\ph(\Ll)&\leq & \ph(\eee)/\ph({\JJJ})~<~n^{K}\pu(\eee)/\pu({\JJJ}) \\
&\sim& n^{K}\pu(\max d_i>\kappa\log n|\Ll)
~<~(1+o(1))n^{-2\kappa+K+1} ~<~ n^{-\kappa},
\end{eqnarray*}
with the second inequality given by \eqref{maxdi3} and Observation~\ref{Obs1};
the ``$\sim$'' by \eqref{mdiL}; and the third inequality by \eqref{nK2} and 
Observation~\ref{Obs2}.\end{proof}

\begin{lemma}\label{Lanemic}
With $\gl = m/|\K|$,
$~\pr(\mbox{$\K_\gl$ anemic}) < \exp[-2n^{2\gd}].$
\end{lemma}
\nin
(We remind once more that $m=m_{_T}$ and recall that ``anemic" was defined in \eqref{anemic}.)

\begin{proof}
With $\gz_x$ the indicator of $\{d_{\K_\gl}(x)<2\eps\log n\}$,
an easy calculation gives
\[
\E \gz_x < n^{-1+2\gd}=:\rho.
\]
(Like \eqref{deg.dev.bd}, this uses the first bound in 
\eqref{eq:ChernoffLower} of Theorem~\ref{T2.1}; of course $\gz_x$ is binomial
while its counterpart in \eqref{deg.dev.bd} was hypergeometric, but the bound applies to both.)

On the other hand the $\gz_x$'s form a read-$r$ family (with corresponding $\psi_i$'s
the indicators $\textbf{1}_{\{A\in \K_\gl\}}$), so Theorem~\ref{GH} gives 
\begin{eqnarray*}
\pr(\mbox{$\K_\gl$ anemic}) &=&\mbox{$\pr(\sum\xi_x\geq 2r n^{2\gd}) $}\\
&<& \exp[-D(2r\rho\|\rho)n/r] < \exp[-2n^{2\gd}]
\end{eqnarray*}
(using
$D(K\rho\|\rho)=(K\log (K/e)+1)\rho +O(\rho^2)$ for fixed $K$ and small $\rho$).
\end{proof}

\begin{lemma}\label{LB}
For sufficiently large $\go \ll \log n$,
\beq{13.9a}
\pr(\max d_{\bH}(x,y) > \go)  < n^{-\gO(\go)}
\enq
(the maximum over distinct vertices $x,y$) and
\beq{13.9b}
\mbox{$\pr(\max_{Z\in \K}\sum_{y\not\in Z}(\gd_y-d_{\bH-Z}(y))^+>\go) 
< n^{-\gO(\go)}.$}
\enq
\end{lemma}
\nin
(Recall that \eqref{13.9b} was promised at \eqref{PhTPx} and used there to show
that $\cap_t\mR^4_t$ is likely (see \eqref{prR2}.)

\begin{proof}
Let
\[
\D = \{\ud:\max d_i< \kappa\log n\},
\]
with $\kappa=\Theta(\go)$ chosen so 
\beq{dDgamma}
\ud\in \D ~\Ra ~ \gc(\ud)> n^{-\go/(2r)}  
\enq
(see Lemma~\ref{Lgamma}), 
and
\[
\N = \{\ud:|\{i:d_i< 2\eps\log n\}|<2 rn^{2\gd}\}
\]
(so $\{\ud(\h)\not\in\N\}=\{\mbox{$\h$ anemic}\}$).
Then Lemma~\ref{LClogn},
its use justified by our assumption that $\go$ is somewhat large, 
says
\beq{PnotD}
\pr(\ud(\bH)\not\in \D)<n^{-\kappa} \,\,\, (=n^{-\gO(\go)}),
\enq
while Lemma~\ref{Lanemic} implies
\beq{PnotN}
\pr(\ud(\bH)\not\in \N) \,\,(=\pr(\mbox{$\bH$ anemic}))~ < \exp[-(2-o(1))n^{2\gd}];
\enq
this follows from 
\[
\exp[-2n^{2\gd}] ~>~ \pr(\mbox{$\K_\gl$ anemic}) ~>~
\pr(|\K_\gl|=m)\pr (\K_\gl \in \mL||\K_\gl|=m)
\pr(\mbox{$\K_\gl$ anemic}|\K_\gl\in \mL_T),
\]
since (i) $\pr(|\K_\gl|=m) \asymp m^{-1/2}$ 
(this is standard and easy);
(ii) $\pr (\K_\gl \in \mL||\K_\gl|=m)=\gb(\K)\more \exp[-n^{\gd}]$ 
(see Corollary~\ref{betacor}); and (iii) on $\{\K_\gl\in \mL_T\}$, $\K_\gl$
is distributed as $\bH$.

By \eqref{PnotD} and \eqref{PnotN},
Lemma~\ref{LB} will follow if we show that, for each
$\ud\in \D\cap\N\cap \Ll$,
\eqref{13.9a} and \eqref{13.9b} hold with $\bH$ replaced by $\bG$ chosen
uniformly from $\K(\ud)$;
so we fix such a $\ud$ and choose $\bG$ in this way.

We will again get at this using the configuration model;
thus we fix the partition $T=\cup T_i$ with $|T_i|=d_i$,
let $\pi:T\ra V $ ($=[n]$) be the corresponding projection,
and for a configuration $F$ use ``$F\in \mG$" (with $\mG$ TBA)
to mean $F$ \emph{is simple and}
$\pi(F)\in \mG$.
Then for a uniformly chosen configuration $\FF$ we have
\beq{PGB}
\pr(\bG\in\mG)=\pr(\FF\in \mG|\FF ~\text{simple}) < \pr(\FF\in \mG)n^{\go/(2r)}
\enq
(with the inequality given by \eqref{dDgamma}).

As earlier, we think of random maps; 
say $\FF=\{\bgs(E_1)\dots \bgs(E_m)\}$, with
$\bgs:E\ra T$ a uniform bijection (and $E=\cup E_i$ as in the proof of Lemma~\ref{Lgamma}).

For \eqref{13.9a}, we have, with $\mG=\{\g: \max d_\g(x,y) > \go\}$,
\beq{PEB1}
\pr(\FF\in\mG) < \C{n}{2}\C{m}{\go}(r(r-1))^\go\left(\frac{\kappa\log n}{mr}\right)^{2\go}
=O(n^{-\go+o(\go)+2}).
\enq
The first three terms of the first bound 
correspond to choosing (i)  $x,y\in [n]$, (ii) $\go$ 
of the $E_i$'s to map to 
preimages of edges containing $x,y$, and
(iii)  elements of these $E_i$'s to map to $T_x$ and $T_y$;
and the last term 
bounds the probability that these choices behave as desired,
using \eqref{pAX} and $\max d_i< \kappa\log n$.
For the final bound recall $\kappa=\Theta(\go) \ll \log n$
($\go=n^{o(1)}$ is enough) and $mr\sim n\log n$.
The combination of \eqref{PGB} and \eqref{PEB1} then gives \eqref{13.9a}
(with $\bG$ in place of $\bH$).  The $\go/(2r)$ in \eqref{dDgamma} and \eqref{PGB} is 
overkill here, but is needed for \eqref{13.9b}.

\mn

For \eqref{13.9b}, fix $Z $ and
let $I=\{i\in V\sm Z:d_i<2\eps \log n\}$---so 
$   
|I|<2rn^{2\gd}
$
since $\ud\in \N$---and let
\beq{bGlast}
\mG = \{\g: |\{A\in \g: A\cap Z\neq\0\neq A\cap I\}|>\go/r\}.
\enq
Then for
$\gS(Z):=\sum_{y\not\in Z}(\gd_y-d_{\bG-Z}(y))^+ >\go$, we must have either
$d_{\bG} (x,y)> (2\eps \log n-\gd_y) $ ($>\go$) 
for some $x\in Z$ and $y\in V\sm (Z\cup I)$---which we have just shown happens
with probability $n^{-\gO(\go)}$---or $\bG\in \mG$
(since, absent such a large codegree, the only $y$'s that can contribute to 
$\gS(Z)$ are those in $I$
and $\gS(Z)$ is at most $r-1$ times the cardinality in \eqref{bGlast}).
We then have (with justification similar to that for \eqref{PEB1})
\begin{eqnarray*}
\pr(\FF\in \mG) &<& \C{m}{\go/r}(r(r-1))^{\go/r}\left(\frac{r\kappa\log n}{mr}\right)^{\go/r}
\left(\frac{|I|\cdot 2\eps\log n}{mr}\right)^{\go/r}\\
&<&n^{-(1+2\gd -o(1))\go/r},
\end{eqnarray*}
and combining with \eqref{PGB} and multiplying by $n^r$ for the choice of $Z$
(and recalling $\go$ is ``sufficiently large'')
gives \eqref{13.9b}.
\end{proof}

\subsection{Back to $\mR$}\label{AppR}

Here, finally, we fill in the remaining promises from Section~\ref{SecR},
namely \eqref{Ri'}
(which means dealing with \eqref{Rg0}-\eqref{Rg2} and \eqref{m3mT}), 
\eqref{PvpZHT} and \eqref{prR3}.

We now use $d_t(\cdot)$ for degree in $\bH_t$ and revert to the ``default''
$m=m_t$.
As in Section~\ref{PLC}, we
work with the generation of $\bH_t$ in \eqref{Generation}:
\beq{bHTbG}
\bH_t=\bH_T\cup \bG,
\enq 
with $\bH_T$ uniform from $\mL_T$
and $\bG$ uniform from $\C{\K\sm\bH_T}{m-m_{_T}}$.
This supports the following little device, 
which will be useful in
the arguments for \eqref{Rg0}, \eqref{PvpZHT} and \eqref{prR3}.

Given $t$,
let $(\bH,\bU)$ be the random pair gotten by choosing
\beq{HandU}
\mbox{$\bH$ uniformly from $\C{\K}{m}$ ($=\K^t$) and then
$\bU$ uniformly from $\C{\bH}{m_{_T}}$.}
\enq
In this section $\bH$ will always be as in \eqref{HandU}
(\emph{not} $\bH_t$ as it was in Section~\ref{Degrees}).
In each application we will have some property $\mG$ for which we would like to show
$\pr(\bH_t\in \mG)$ is small, and
(\emph{slightly} echoing Section~\ref{PLC})
will exploit information gotten by reversing the order in \eqref{HandU};
that is, by choosing
\[
\mbox{$\bU$ uniformly from $\C{\K}{m_{_T}}$ and $\bH$ uniformly from
$\{\h\in  \K^t:\h\supseteq \bU\}$.}
\]
Thinking of the process in this way and setting
\beq{theta}
\Theta=\pr(\bH\in \mG,\bU\in \mL),
\enq
we have
\beq{theta=}
\Theta ~=~ \pr(\bU\in \mL)\pr(\bH\in \mG|\bU\in \mL)
~=~\gb(\K)\pr(\bH_t\in \mG)
\enq
(since on  $\{\bU\in \mL\}$, $\bH$ is distributed as $\bH_t$),
which we will combine with upper bounds on $\Theta$
based on the viewpoint in \eqref{HandU}.

\mn

\begin{proof}[Proof of \eqref{Ri'}]
Recall this says that 
w.h.p.\ $\bH_t$ satisfies \eqref{Rg0}-\eqref{Rg2} and \eqref{m3mT}
for all $t\leq T$.

\mn

For \eqref{Rg0} 
we may appeal to \cite{AsSh}:
as shown there---see the paragraph containing (131)---the probability that 
$\bH$ as in \eqref{HandU} violates \eqref{Rg0} is $e^{-\gO(n)}$.
(Precisely, with $\theta = (\log n)^{-1/3}$, it is shown that
$e^{-\gO(n)}$ bounds the probability that $d_{\bH}(x)\neq (1\pm \theta) D_{\bH}$
for at least $\theta n$ vertices $x$.)

Then with $\mG =\{\h\sub \K:\mbox{$\h$ violates \eqref{Rg0}}\}$ (and $\Theta$ as in \eqref{theta}), 
we have 
$\Theta< \pr(\bH\in \mG) = \exp[-\gO(n)]$ and, using \eqref{theta=} and Corollary~\ref{betacor},
\[
(\pr (\mbox{$\bH_t$ violates \eqref{Rg0}})=) \,\,\,
\pr(\bH_t\in \mG) = \Theta/\gb(\K)= \exp[-\gO(n)].
\]

\mn

For \eqref{Rg1} and \eqref{m3mT} (with the lower bound in the former contained in
the latter, as observed following \eqref{m3mT}), we use \eqref{bHTbG}, noting that,
given $\bH_T$, each 
$d_{\bG}(x)$ is hypergeometric with 
\beq{mux}
\mu_x:=
\E d_{\bG}(x) = \frac{D_{\K}-d_T(x)}{|\K|-m_{_T}}(m-m_{_T})
\sim (1-m_{_T}/m)D_m
\enq
(since $D_m= D_{\K}m/|\K|$).

For the upper bound in \eqref{Rg1} we first note that 
$\gD_{\bH_T} = O(\log n)$ w.h.p.\ by Lemma~\ref{LClogn}
(in which, recall, $\bH$ was $\bH_T$).
Then for $\bG$, Theorem~\ref{Cher'} gives (very wastefully but we don't care)
\[
\pr(\exists x,t \,\, d_t(x) > 3rD_m) < n^{r+1}\exp[-3rD_m\log (3r/e)] =o(1).
\]

For \eqref{m3mT} it will be enough to consider $\bG$.
Here \eqref{mux} and $D_m\sim  (m/m_{_T})\log n$ imply $\mu_x\more (2r-1)\log n$, and then
Theorem~\ref{T2.1}
(see the first bound in \eqref{eq:ChernoffLower}) gives
(for any $x$)
\beq{prdbGx}
\pr(d_{\bG}(x)< 2\eps D_m) <   \exp[-\mu_x\vp(-1+2\eps D_m/\mu_x)] 
< n^{-(2r-1)+O(\eps\log (1/\eps))}.
\enq

For \eqref{Rg2}, we again use \eqref{bHTbG}, 
noting to begin that Lemma~\ref{LB} says that w.h.p.\ $\bH_T$ has maximum codegree $O(1)$.
For $\bG$ we again use Theorem~\ref{Cher'} (with plenty of room):
each $d_{\bG}(x,y)$ is hypergeometric with mean

\[
\frac{d_{\K}(x,y)-d_T(x,y)}{|\K|-m_{_T}}(m-m_{_T})    < rD_m/n,
\]
whence $\pr(\max d_t(x,y) > CD_m/n) < n^{-(r+1)}$ for a suitable fixed $C$.
\end{proof}

\begin{proof}[Proof of \eqref{PvpZHT}]
Given $Z\in \K$, set
\[
\mG=\{\h\sub \K:\barvp_Z(\h)<\eta \gb(\K)\},
\]
so \eqref{PvpZHT} is
\beq{PvpZHT'}
\pr(\bH_t\in \mG)<\eta.
\enq

Let $(\bH,\bU)$ and $\Theta$ again be as in \eqref{HandU}
and \eqref{theta},
and set $\bG=\bH-Z$.
Note that, as $\barvp_Z(\h)$ depends only on $|\h|$ and $\h-Z$,
membership of $\bH$ in $\mG$ is decided by $\bG$.
We now think of choosing first 
$\bG$ and then $\bH\sm\bG$ and $\bU$.
The law of $\bG$ plays no role here; what matters is that $\bH\sm\bG$ is
uniform from $\mJ_Z(\bH)$ ($=\{\h'\in \K^t: \h'-Z=\bG\}$; see \eqref{mKZH}),
so that (essentially by definition; see \eqref{vpZh}, \eqref{barvpZh})
\[
\pr(\bU\in \mL|\bG) =\barvp_Z(\bH).
\]
Thus
\[   
\Theta = \pr(\bH\in \mG)\pr(\bU\in \mL|\bH\in\mG) <
\pr(\bH\in \mG)\cdot \eta\gb(\K).
\]   
We then sacrifice the first factor on the r.h.s.\ and combine with \eqref{theta=}
to get \eqref{PvpZHT'}.
(The sacrifice in this case is substantial, but we aren't asking much and can afford it.)\end{proof}

\begin{proof}[Proof of \eqref{prR3}]
(Recall this said
$
\pr(\bH_t\not\in \mR^3) < \exp[-(1-o(1))n^{2\gd}]
$
$\forall t\leq T$.)
Set  
$q=m/|\K|$ and (as in $\mR^3$ and Lemma~\ref{Lanemic} resp.) $\ga=m_{_T}/m$
and $\gl=m_{_T}/|\K|$;
so $\K_\gl = (\K_q)_\ga$.
From Lemma~\ref{Lanemic} and the definition of $\mR^3$ (and the fact that 
on $\{|\K_q|= m\}$, $\K_q$ is distributed as $\bH$), we have
\begin{eqnarray*}
\exp[-2n^{2\gd}] &>& \pr(\mbox{$\K_\gl$ anemic}) \\
&>& 
\pr(|\K_q|= m)\pr(\bH \not\in \mR^3)\exp[-(1-o(1))n^{2\gd}],
\end{eqnarray*}
which, since $\pr(|\K_q|= m)\asymp m^{-1/2}$, implies 
$\pr(\bH \not\in \mR^3) < \exp[-(1-o(1))n^{2\gd}]$.
We then set
$\mG =\{\h\sub \K:\h\not\in \mR^3\}$ to obtain, as in the above treatment of \eqref{Rg0},
\begin{eqnarray*}
\pr (\bH_t\not\in \mR^3)&=&
\pr(\bH_t\in \mG) ~=~ \Theta/\gb(\K) \\
&\leq &
\pr(\bH \not\in \mR^3)/\gb(\K)
~<~\exp[-(1-o(1))n^{2\gd}].
\end{eqnarray*}
\end{proof}


\begin{thebibliography}{99}


\bibitem{AS} N. Alon and J. Spencer,
{\em The Probabilistic Method}, Wiley, New York, 2000.



\bibitem{ALWZ}
R.\ Alweiss, S.\ Lovett, K.\ Wu and J.\ Zhang, 
Improved bounds for the sunflower lemma,
\emph{Proc.\ 52nd Annual ACM SIGACT Symp.\ Th.\ of Computing} (2020), 624-630.

\bibitem{BJ}
J.\ van den Berg and J.\ Jonasson
A BK inequality for randomly drawn subsets of fixed size,
Prob.\ Th.\ Related Fields \textbf{154} (2012), 835-844.

\bibitem{BK} J. van den Berg and H. Kesten, Inequalities with
applications to percolation and reliability, \emph{J. Appl. Probab.}\
\textbf{22} (1985), 556-569.

\bibitem{Boll79}
B.\ Bollob\'as, A probabilistic proof of an asymptotic formula for the
number of labelled regular graphs, Preprint Series, Matematisk Institut,
Aarhus Universitet, 1979.




\bibitem{BT}
B. Bollob\'as and A. Thomason,
Random graphs of small order,
pp. 47-97 in \emph{Random graphs '83},
North-Holland, Amsterdam, 1985.


\bibitem{BTth} B. Bollob\'as and A. Thomason, Threshold functions,
{\em Combinatorica} {\bf 7} (1987), 35-38.





\bibitem{CFMR}
C. Cooper, A. Frieze, M. Molloy and B. Reed,
Perfect matchings in random $r$-regular, $s$-uniform hypergraphs,
{\em Combin. Probab. Comput.} {\bf 5} (1996), 1-14.



\bibitem{Erdos}
P. Erd\H{o}s,
On the combinatorial problems which I would most like to see solved,
{\em Combinatorica} {\bf 1} (1981), 25-42.

\bibitem{Erdos-Lovasz}
P.\ Erd\H{o}s and L.\ Lov\'asz,
Problems and results on 3-chromatic hypergraphs and some related questions,
\emph{Coll.\ Math.\ Soc.\ J.\ Bolyai} {\bf 10} (1974), 609-627.

\bibitem{ER}
P. Erd\H{o}s and A. R\'enyi, On the evolution of random graphs,
{\em Publ. Math. Inst. Hungar. Acad. Sci.} {\bf 5} (1960), 17-61.

\bibitem{ERPM}
P. Erd\H{o}s and A. R\'enyi, On the existence of a factor of degree one of
a connected random graph,
{\em Acta Math. Acad. Sci. Hungar.} {\bf 17} (1966), 359-368.

\bibitem{ES}
P. Erd\H{o}s and J. Spencer, Lopsided Lov\'asz local lemma and latin transversals,
\emph{Disc.\ Appl.\ Math.} \textbf{30} (1991), 151-154.

\bibitem{FKNP}
K.\ Frankston, J.\ Kahn, B.\ Narayanan and J.\ Park,
Thresholds vs.\ fractional expectation-thresholds, submitted. 
 arXiv:1910.13433v2 [math.CO]


\bibitem{FJ}
A. Frieze and S. Janson,
Perfect matchings in random $s$-uniform hypergraphs,
{\em Random Structures \& Algorithms} {\bf 7} (1995), 41-57.

\bibitem{Frieze-Karonski}
A.\ Frieze and M.\ Karonski, \emph{Introduction to Random Graphs,}
Cambridge Univ.\ Pr., Cambridge, 2016.




\bibitem{GLSS}
D.\ Gavinsky, S.\ Lovett, M.\ Saks and S.\ Srinivasan,
A Tail Bound for Read-k Families of Functions,
\emph{Random Structures \& Algorithms} \textbf{47} (2015), 99-108.


\bibitem{Grimmett}
G.\ Grimmett, \emph{Percolation}, Springer, Berlin, 1999.

\bibitem{Harris} T.E. Harris,
A lower bound on the critical probability in a certain percolation process,
\emph{Proc. Cam. Phil. Soc.} \textbf{56} (1960), 13-20.

\bibitem{Heckel}
A. Heckel, Random triangles in random graphs,
\emph{Random Struct.\ Alg.}, to appear.
arXiv:1802.08472 [math.CO]




\bibitem{JLR} S. Janson, T. \L uczak and A. Ruci\'nski,
{\em Random Graphs}, Wiley, New York, 2000.

\bibitem{JKV}
A. Johansson, J. Kahn and V. Vu, Factors in random graphs
{\em Random Structures \& Algorithms} {\bf 33} (2008), 1-28.

\bibitem{AsSh} J. Kahn,
Asymptotics for Shamir's Problem, submitted.
arXiv:1909.06834v1 [math.CO] 


\bibitem{KK}
J. Kahn and G. Kalai, Thresholds and expectation thresholds,
{\em Combinatorics, Probab. Comput.} {\bf 16} (2007), 495-502.

\bibitem{Kim} J.H. Kim,
Perfect matchings in random uniform hypergraphs,
{\em Random Structures \& Algorithms} {\bf 23} (2003), 111-132.



\bibitem{McKay}
B. D. McKay,
Asymptotics for symmetric 0-1 matrices with prescribed row sums,
\emph{Ars Combinatoria}, \textbf{19A} (1985), 15-25.

\bibitem{MW} B.D.\ McKay and N.C.\ Wormald,
Asymptotic enumeration by degree sequence of graphs with degrees $o(n^{1/2})$,
\emph{Combinatorica} \textbf{11} (1991), 369-382.



\bibitem{Riordan}
O. Riordan, Random cliques in random graphs,
arXiv:1802.01948 [math.CO]


\bibitem{Rucinski2} A. Ruci\'nski,
Open problem, in {\em Random Graphs (2)}, Proceedings, Pozna\'n,
1989, A.M. Frieze and T. {\L}uczak, eds.,
John Wiley \& Sons, New York, 284.


\bibitem{SS}
J. Schmidt and E. Shamir,
A threshold for perfect matchings in random $d$-pure
hypergraphs, {\em Disc. Math.} {\bf 45} (1983), 287-295.


\bibitem{Talagrand}
M.\ Talagrand, Are many small sets explicitly small?, pp.\ 13-35 in
\emph{Proc.\ STOC `10}, ACM, Cambridge, 2010.


\bibitem{Wormald}
N.C. Wormald, Models of random regular graphs, pp.\ 239-298 in
\emph{Surveys in Combinatorics, 1999}, J.D.\ Lamb and D.A.\ Preece, eds.,
London Math.\ Soc.\ Lecture Note Series \textbf{276}, Cambridge Univ. Press, Cambridge, 1999.


\end{thebibliography}
\end{document}